\documentclass[11pt]{article}

\usepackage{epsfig,amsmath,amssymb,graphics,verbatim,amsfonts,subfigure,psfrag,amsthm,bm}
\usepackage[a4paper,margin=1in]{geometry}
\usepackage{color}
\usepackage{tikz}
\usetikzlibrary{matrix,arrows,calc}
\usetikzlibrary{decorations.pathmorphing}
\usetikzlibrary{decorations.markings}
\usepgflibrary{fpu}
\usepackage{dsfont} 

\usepackage{enumitem}
\setlist{parsep=1pt} % slightly decrease enumerate/itemize separation

\usepackage[percent]{overpic}

\numberwithin{equation}{section}

\let\oldbibliography\thebibliography
\renewcommand{\thebibliography}[1]{%
  \oldbibliography{#1}%
  \setlength{\itemsep}{0.5mm}%
}

\newtheorem{thm}{Theorem}[section]
\newtheorem{cor}[thm]{Corollary}
\newtheorem{prop}[thm]{Proposition}
\newtheorem{lem}[thm]{Lemma}

\theoremstyle{definition}

\newtheorem{examplex}[thm]{Example}
\newenvironment{ex}
  {\pushQED{\qed}\examplex }
  {\popQED\endexamplex}
\newtheorem{remarkx}[thm]{Remark}
\newenvironment{rmk}
  {\pushQED{\qed}\remarkx}
  {\popQED\endremarkx}
\makeatletter

\def\captionheadfont@{\scshape}
\def\captionfont@{\small}
\long\def\@makecaption#1#2{%
  \setbox\@tempboxa\vbox{\color@setgroup
    \advance\hsize-3pc\noindent
    \captionfont@\captionheadfont@#1\@xp\@ifnotempty\@xp
        {\@cdr#2\@nil}{.\captionfont@\upshape\enspace#2}%
    \unskip\kern-3pc\par
    \global\setbox\@ne\lastbox\color@endgroup}%
  \ifhbox\@ne % the normal case
    \setbox\@ne\hbox{\unhbox\@ne\unskip\unskip\unpenalty\unkern}%
  \fi
  \ifdim\wd\@tempboxa=\z@ % this means caption will fit on one line
    \setbox\@ne\hbox to\columnwidth{\hss\kern-3pc\box\@ne\hss}%
  \else % tempboxa contained more than one line
    \setbox\@ne\vbox{\unvbox\@tempboxa\parskip\z@skip
        \noindent\unhbox\@ne\advance\hsize-3pc\par}%
\fi
  \ifnum\@tempcnta<64 % if the float IS a figure...
    \addvspace\abovecaptionskip
    \moveright 1.5pc\box\@ne
  \else % if the float IS NOT a figure...
    \moveright 1.5pc\box\@ne
    \nobreak
    \vskip\belowcaptionskip
  \fi
\relax
}

\makeatother

\def\R{\mathbb{R}}
\def\Q{\mathbb{Q}}
\def\T{\mathbb{T}}
\def\P{\mathbb{P}}
\def\E{\mathbb{E}}

\def\N{\mathbb{N}}
\def\Z{\mathbb{Z}}
\def\I{\infty}
\def\Id{{\textnormal{Id}}}

\def\txtd{{\textnormal{d}}}
\def\txte{{\textnormal{e}}}
\def\txti{{\textnormal{i}}}

\def\txtD{{\textnormal{D}}}

\DeclareMathOperator{\e}{e}             % Napier's number

\newcommand{\be}{\begin{equation}}
\newcommand{\ee}{\end{equation}}
\newcommand{\bea}{\begin{eqnarray}}
\newcommand{\eea}{\end{eqnarray}}
\newcommand{\beann}{\begin{eqnarray*}}
\newcommand{\eeann}{\end{eqnarray*}}
\newcommand{\benn}{\begin{equation*}}
\newcommand{\eenn}{\end{equation*}}

\DeclareMathSymbol{\leqsymb}{\mathalpha}{AMSa}{"36}
\def\leqs{\mathrel\leqsymb}
\DeclareMathSymbol{\geqsymb}{\mathalpha}{AMSa}{"3E}
\def\geqs{\mathrel\geqsymb}
\def\ra{\rightarrow}
\def\I{\infty}

\DeclareMathOperator{\Var}{Var}

\newcommand{\cA}{{\mathcal A}}  % calligraphic A
\newcommand{\cC}{{\mathcal C}}  % calligraphic C
\newcommand{\cD}{{\mathcal D}}  % calligraphic D
\newcommand{\cF}{{\mathcal F}}  % calligraphic F
\newcommand{\cG}{{\mathcal G}}  % calligraphic G
\newcommand{\cI}{{\mathcal I}}  % calligraphic I
\newcommand{\cO}{{\mathcal O}}  % calligraphic O
\newcommand{\cQ}{{\mathcal Q}}  % calligraphic Q
\newcommand{\cR}{{\mathcal R}}  % calligraphic R
\newcommand{\cS}{{\mathcal S}}  % calligraphic S
\newcommand{\cT}{{\mathcal T}}  % calligraphic T
\newcommand{\cU}{{\mathcal U}}  % calligraphic U
\newcommand{\cW}{{\mathcal W}}  % calligraphic W
\newcommand{\fs}{{\mathfrak s}}  % Fraktur s
\newcommand{\fraks}{{\mathfrak s}}  % Fraktur s
\newcommand{\fT}{{\mathfrak T}}  % Fraktur T
\newcommand{\Hs}{\ensuremath{\mathfrak{H}}}
\newcommand{\Xs}{\ensuremath{\mathfrak{X}}}

\newcommand{\bz}{{\bar z}}  % bar z
\newcommand{\bx}{{\bar x}}  % bar x
\newcommand{\bt}{{\bar t}}  % bar t
\def\fThat{\widehat{\mathfrak T}}

\def\rhocrit{\rho_{\rm{c}}}
\def\figref#1{Figure~\ref{#1}}

\newcommand{\decapprox}[1]{{}\approx #1}
\DeclareMathSymbol{\leqsymb}{\mathalpha}{AMSa}{"36}
\renewcommand{\leq}{\mathrel\leqsymb}
\renewcommand{\leqs}{\mathrel\leqsymb}
\DeclareMathSymbol{\geqsymb}{\mathalpha}{AMSa}{"3E}
\renewcommand{\geq}{\mathrel\geqsymb}
\renewcommand{\geqs}{\mathrel\geqsymb}
\begin{document}
\title{Model Spaces of Regularity Structures \\
for Space-Fractional SPDEs}
\author{Nils Berglund and Christian Kuehn}
\date{February 25, 2017}   

\maketitle

\begin{abstract}
We study model spaces, in the sense of Hairer, for stochastic partial
differential equations involving the fractional Laplacian. We prove that the
fractional Laplacian is a singular kernel suitable to apply the theory of
regularity structures. Our main contribution is to study the dependence of the
model space for a regularity structure on the three-parameter problem
involving the spatial dimension, the polynomial order of the nonlinearity, and
the exponent of the fractional Laplacian. The goal is to investigate the growth
of the model space under parameter variation. In particular, we prove several
results in the approaching subcriticality limit leading to universal growth
exponents of the regularity structure. A key role is played by the viewpoint
that model spaces can be identified with families of rooted trees. Our
proofs are based upon a geometrical construction similar to Newton polygons for
classical Taylor series and various combinatorial arguments. We also present
several explicit examples listing all elements with negative homogeneity by
implementing a new symbolic software package to work with regularity structures.
We use this package to illustrate our analytical results and to obtain new
conjectures regarding coarse-grained network measures for model spaces. 
\end{abstract}

\leftline{\small{\bf Submitted version}}
\leftline{\small 2010 {\it Mathematical Subject Classification.\/} 
% 60 Probability theory and stochastic processes
60H15, %Stochastic partial differential equations
% 35 Partial differential equations
35R11 (primary), %Fractional partial differential equations
% 05 Combinatorics 
% 05C Graph theory 
05C05,   %Trees
% 82 Statistical mechanics, structure of matter
82B20 (secondary). %lattice systems (Ising, dimer, Potts, etc.) and systems
%on graphs
}

\section{Introduction}
 
Our main starting point in this work are stochastic partial differential equations 
(SPDEs) of the form
\be
\label{eq:AC}
\partial_t u = \Delta^{\rho/2} u+f(u)+\xi\;,\qquad
\partial_t:=\frac{\partial}{\partial t}\;,
\ee
where $\Delta^{\rho/2}:=-(-\Delta)^{\rho/2}$ is the \emph{fractional Laplacian} for 
$\rho\in(0,2]$, $f$ is a polynomial defined in \eqref{eq:poly}, $\xi=\xi(t,x)$ is the 
noise (which we shall often take as space-time white noise as discussed below), $u=u(t,x)$ 
and $(t,x)\in [0,+\I) \times \T^d$, where $\T^d$ is the unit torus in $\R^d$ so that we 
work with periodic boundary conditions. We also write $f$ explicitly as
\be
\label{eq:poly}
f(u):=\sum_{j=0}^Na_j u^j\;,
\ee
where the degree $N$ will be restricted, depending upon the dimension $d$. The
technical development of this paper starts in Section~\ref{sec:regstruct}. Here
we outline the origins of the questions of this work and provide a formal
overview of our results.\medskip

There are several motivations to study~\eqref{eq:AC}. First of all, the case
$\alpha=2$, corresponding to the classical Laplacian, has been studied in great
detail for the deterministic system ($\xi\equiv 0$), under various names for
different polynomial nonlinearities such as the Allen--Cahn
equation~\cite{AllenCahn}, the Nagumo equation~\cite{Nagumo}, the
Fisher--Kolmogorov--Petrowski--Piscounov
equation~\cite{Fisher,KolmogorovPetrovskiiPiscounov}, the Ginzburg--Landau
equation~\cite{CrossHohenberg} or as normal forms (or modulation/amplitude
equations) for bifurcations of PDEs~\cite{Hoyle}. Recently, there has been
considerable interest to extend the scope to \emph{either} the stochastic case
($\xi\not\equiv 0$) \emph{or} to include other differentiable operators, such as
the fractional Laplacian, which is a nonlocal operator;
see~\cite{ChenKimSong,AchleitnerKuehn1} or Section~\ref{sec:kernels} for one
precise definition of the fractional Laplacian. Hence, it is a natural
theoretical question, how stochastic terms and nonlocal operators can be
combined. For very rough stochastic driving terms $\xi$ already existence and
regularity questions for~\eqref{eq:AC} are nontrivial. Indeed, pursuing the
classical route to re-write~\eqref{eq:AC} in a mild
formulation~\cite{DaPratoZabczyk} using the semigroup $t\mapsto
\smash{\txte^{t\Delta^{\rho/2}}}$ can lead to a fixed-point problem, where the
products in the polynomial~\eqref{eq:poly} are not well-defined as $u$ may only
have the regularity of a generalized function (or distribution). This issue not
only appears for the classical stochastic Allen--Cahn equation ($\alpha=2$,
$\xi$ space-time white noise, $d\in\{2,3\}$, $N=3$) but also in many other SPDEs
such as the Kardar--Parisi--Zhang equation~\cite{KardarParisiZhang,Hairer2}, the
$\Phi^4_3$-model~\cite{Hairer3}, and several other SPDEs. To overcome this
problem for different classes of SPDEs, Martin Hairer developed a unified theory
of \emph{regularity structures} in combination with a renormalization
scheme~\cite{Hairer1}. The basic idea is to consider an abstract algebraic
structure, the regularity structure, in which we may solve a suitable
fixed-point problem. Furthermore, this solution procedure is compatible with
taking limits of a smoothed version of the original problem and reconstructing a
\lq\lq physical\rq\rq\ solution from objects in the regularity
structure~\cite{Hairer1}. A key ingredient of the fixed-point argument is that
under a parabolic space-time scaling, the heat semigroup associated to the
classical Laplacian is regularizing of order two, i.e., we essentially  gain two
derivatives after a convolution of a function/distribution with the heat kernel.
This smoothing effect compensates the rough driving by $\xi$ for certain,
so-called subcritical, nonlinearities. A related theory achieving the definition
of low-regularity products of functions uses paracontrolled
distributions~\cite{Bony,Gubinelli}, developed by Gubinelli and co-authors, and
has also been applied to several classes of
SPDEs~\cite{GubinelliTindel,GubinelliImkellerPerkowski}. Also in this theory,
one makes use of the smoothing properties of differential operators in the SPDE.

\begin{rmk}
In this paper, we focus on the theory of regularity structures. However, since
we study a space of functions, the space of modelled distributions, which can be 
used to expand the solution of the SPDE~\eqref{eq:AC} as a series, it is expected 
that our results could be re-interpreted in the context of paracontrolled 
calculus in future work.
\end{rmk}

From the theory of fractional differential Laplace operators, or more generally
Riesz--Feller operators, and associated ideas developed in the context of L\'evy
processes and anomalous diffusion, it is known that the fractional Laplacian has
\emph{less} smoothing properties than the classical Laplacian. Hence, it is
natural to ask, how this restricts the class of SPDEs tractable via regularity
structures.\medskip

A second key motivation to study~\eqref{eq:AC} arose out of recent
work~\cite{BerglundKuehn} by the authors of this paper, where the key object of
study was the FitzHugh--Nagumo SPDE
\begin{align}
\nonumber
\partial_t u &= \Delta u + u(1-u)(u-p)-v+\xi\;,\\
\partial_t v &= g(u,v)\;, 
\label{eq:FHN}
\end{align}
where $p\in[0,\frac12]$ is a parameter and $g(u,v):\R\times \R^m\ra \R^m$ 
is linear. The main problems to apply the theory of regularity structures
to~\eqref{eq:FHN} are to deal with multiple components and especially the
missing spatial regularization properties as the $v$-components do not contain a
spatial differential operator. In~\cite{BerglundKuehn}, this problem is overcome
by introducing a spatially non-smoothing operator for the $v$-components and
then working through the required Schauder estimates for this operator again by
modifying the relevant parts of~\cite{Hairer1}. One may argue that it could be
more convenient to try to follow the classical strategy developed for hyperbolic
conservation laws~\cite{Dafermos} and introduce a viscous regularization 
\begin{align}
\nonumber
\partial_t u &= \Delta u + u(1-u)(u-p)-v+\xi\;,\\
\partial_t v &= \delta_v\Delta v + g(u,v)\;, 
\label{eq:FHNviscous}
\end{align}
prove existence and regularity for~\eqref{eq:FHNviscous} and then consider the
limit $\delta_v\ra 0$. Albeit appealing, this approach does not seem to work 
well in combination with regularity structures as $\delta_v \Delta$ is a 
regularizing operator of order two for any $\delta_v>0$ but has no spatial 
regularization properties for $\delta_v=0$. Hence, one idea is to replace 
$\delta_v\Delta$ by $\Delta^{\rho/2}$ and consider the limit $\rho\ra 0$ 
instead. This strategy seems more suited to work with regularity structures
but we are still lacking a full understanding of \emph{one-component} 
reaction--diffusion SPDEs with polynomial nonlinearities and involving the 
fractional Laplacian. This paper provides the first steps to fill this
gap.\medskip 

Additional motivation arises from the theory of regularity structures
itself. The theory does require certain nonlocal operators, even when working with the
classical Laplacian, so working with general nonlocal operators seems natural
and could lead to a better understanding of regularity structures. Furthermore,
the algebraic results on regularity structures recently announced
in~\cite{Hairer4}, developed simultaneously to the present work and available as
first versions in~\cite{BrunedHairerZambotti,ChandraHairer16}, are also relevant
as one may ask, how the viewpoint of rooted trees as regularity structure
elements may be exploited further. We remark that we are going to study a
problem, where the exponent of the fractional Laplacian $\rho$ allows us to
approach the border between subcritical SPDEs with finite renormalization group
(i.e.~super-renormalizable SPDEs)~\cite{Hairer2} and infinite-dimensional
renormalization groups.

From the perspective of dynamical aspects of the SPDEs~\eqref{eq:AC}
and~\eqref{eq:FHN}, it is known that solution properties can change
considerably if external noise is added~\cite{ArnoldSDE,Gardiner} or 
standard diffusion is replaced by a jump 
process~\cite{KlagesRadonsSokolov,MetzlerKlafter}. Hence, combining and 
comparing these two aspects is definitely going to lead to new 
applications.\medskip

In this paper, we focus on the \emph{parametric problem} for building the
regularity structure, i.e., we only study the model space of the structure,  not
yet the renormalization procedure, which might simplify anyhow due to the very
recent results in~\cite{Hairer4,BrunedHairerZambotti,ChandraHairer16}; however,
a detailed renormalization study of~\eqref{eq:AC} is still expected to reveal
interesting effects, e.g., in the context of large
deviations~\cite{HairerWeber,BerglundDiGesuWeber}. There are three key
parameters for~\eqref{eq:AC}, whose influence we would like to study:

\begin{itemize}
 \item \textit{the spatial dimension} $d\in\N$,
 \item \textit{the polynomial degree} $N\in\N$,
 \item \textit{the fractional exponent} $\rho\in(0,2]$,
\end{itemize}

\noindent
where the choice of interval for $\rho$ is motivated by the regularization of
the FitzHugh--Nagumo problem as well as the representative horizontal cut 
through the Feller--Takayasu diamond~\cite{KlagesRadonsSokolov}; in particular, 
our results also hold for Riesz--Feller operators with asymmetry parameter 
$\theta$ and $|\theta|\leq \min\{\alpha,2-\alpha\}$ but we do not spell the 
results out to simplify the notation. There seems currently to be no detailed
study available covering the parametric dependence of regularity structures, so
we hope that giving this dependence more explicitly for a key example such 
as~\eqref{eq:AC} helps to solidify the general understanding of regularity 
structures. The main results of our work can be summarized in a non-technical 
form as follows:

\begin{enumerate}
\item 	We verify that the fractional Laplacian falls within the class of
singular integral operators considered by Hairer. We determine for which
parameter values the SPDE~\eqref{eq:AC} fulfils the local subcriticality
condition, which is necessary for the application of regularity structures to
construct solutions; the main condition for space-time white noise is
\[
\rho>\rho_c=d\frac{N-1}{N+1}\;;
\]
see Section~\ref{sec:kernels} and Theorem~\ref{thm:genACsub1}.

\item 	We study the model space $\cT_F$
and its dependence on the three parameters $N,d,\rho$ analytically. There are 
several interesting results for space-time white-noise:
\begin{itemize}
\item[(a)] We relate the number of negative homogeneities to solutions of
constrained Diophantine equations for rational $\rho$;
see Proposition~\ref{prop:Dio}.

\item[(b)] We prove that the number of negative homogeneities in the
index set diverges like 
$(\rho-\rhocrit)^{-1}$ 
as $\rho$ approaches the subcriticality limit from above; see
Theorem~\ref{thm:hF}. 

\item[(c)] We prove that the number of negative-homogeneous elements of the
model space diverges in the subcriticality limit like  
$(\rho - \rhocrit)^{3/2}\exp (\beta_N d (\rho-\rhocrit)^{-1})$
as $\rho\searrow \rho_c$ with an explicitly computable constant $\beta_N$;
see Theorem~\ref{thm:cF2} and Theorem~\ref{thm:Nbigger2}.

\item[(d)] We determine statistical properties of the
negative-homogeneous elements of the model space viewed as rooted trees, such
as the asymptotic relative degree distribution (Proposition~\ref{prop:lawDj})
and the homogeneity distribution (Proposition~\ref{prop:lawH}). In particular,
we show that all trees are obtained by pruning a regular tree of degree $N+1$
of at most $N-1$ edges.
\end{itemize}
For a more general noise, we have the following result:
\begin{itemize}
 \item[(e)]  For any value of $(\rho,N,d)$, there exists a choice of noise
such that the SPDE~\eqref{eq:AC} is subcritical and its model space has elements
of homogeneity of $0^-$, i.e., $0-\kappa$ for any sufficiently small $\kappa$;
see Proposition~\ref{prop:bifalways}.
\end{itemize}    
 \item We study numerically the statistical properties of elements in $\cT_F$
viewed as rooted trees using graph-theoretical measures. We implement a symbolic
software package to compute with the rooted trees. Based upon computing several
benchmark examples, we conjecture a number of statistical limit graph
properties. These conjectures include the existence of limits of coarse-grained
graph measures as well as the existence of a limiting probability distribution
of negative homogeneities formed (for $\rho\in \Q$) similar to a fractal
construction; see Section~\ref{sec:compute}.
\end{enumerate}

In summary, the approach towards the theory of regularity structures in this 
paper is quite different from previous works using the theory. Here we focus
on the parametric structure of the model space and SPDEs approaching the 
subcriticality condition in a parameter limit. This approach is similar to
strategies employed in statistical physics to capture \lq\lq universal\rq\rq\
exponents as well as to parameter studies in dynamical systems. This view
reveals that there is a lot to be learned just by studying the model space as an
object itself.\medskip

The paper is structured as follows: In Section~\ref{sec:regstruct}, we review
some basics about regularity structures and fix the notation. In
Section~\ref{sec:kernels} we provide the details for the singular kernels, 
as introduced by Hairer, in the context of the fractional Laplacian. In the main
Section~\ref{sec:canreg}, we construct the model space $\cT_F$
for~\eqref{eq:AC}, we determine the subcriticality boundary, and we study the
growth of the number of elements in $\cT_F$. In Section~\ref{sec:statreg},
we derive several statistical properties of the model space near the
subcriticality boundary. In Section~\ref{sec:compute}, we present the
computations carried out by the new symbolic computation package.

\medskip
\noindent
\textbf{Notations:} If $x\in\R$, then $\lfloor x \rfloor$ denotes the
largest integer less than or equal to $x$. We write $|x|$ to denote either the
absolute value of $x\in\R$ or the $\ell^1$-norm of $x\in\R^{d}$, while $\|x\|$
denotes the Euclidean norm of $x\in\R^{d}$. If $A$ is a finite set, then
$|A|$ stands for the cardinality of $A$.
We use the notation $\text{supp}(f)$ for the support of a function $f$, and
$f\asymp g$ to indicate that $f/g$ is bounded uniformly above and below by
strictly positive constants. For the identity mapping on a set or space we
write $\Id$.

\medskip
\noindent
\textbf{Acknowledgments:} We would like to thank Romain Abraham, Marie
Albenque and Kilian Raschel for advice on the combinatorics of trees. CK has
been supported by the {Volkswagen\-Stiftung} via a Lichtenberg Professorship.

\section{Regularity Structures}
\label{sec:regstruct}

We briefly recall the basic notions from~\cite{Hairer1} to fix the notation.
Furthermore, this section separates out the parts of the theory of regularity
structures that are required to consider the fractional SPDE~\eqref{eq:AC},
i.e., to build the regularity structure associated to~\eqref{eq:AC}. Let
$\mathfrak{T}=(\cA,\cT,\cG)$ denote a \emph{regularity structure}. The
\emph{index set} $\cA\subset \R$ is locally finite, bounded below and with
$0\in \cA$, the \emph{model space} $\cT=\oplus_{\alpha\in\cA}\cT_\alpha$ is a
graded vector space with Banach spaces $\cT_\alpha$, and $\cG$ is the
\emph{structure group} of linear operators acting on $\cT$, such that for
every $\Gamma \in \cG$, $\alpha\in \cA$, $\tau\in \cT_\alpha$ one has
\be
\label{eq:lowerorder}
\Gamma \tau-\tau\in \bigoplus_{\beta <\alpha} \cT_\beta\;.
\ee
Roughly speaking, one has to view $\cT$ as containing abstract symbols
representing basic functions or distributions, while $\cG$ links together
different basis points for general abstract expansions of functions or
distributions. Two basic examples are the polynomial regularity structure (see
Example~\ref{ex:regpoly} below) and rough paths~\cite{FrizHairer}. It is
standard to require that $\cT_0$ is isomorphic to $\R$ with unit vector ${\bf
1}$ as well as $\Gamma {\bf 1}={\bf 1}$ for all $\Gamma\in \cG$. It will become
clear later on that the index $\alpha$ can be associated to \lq\lq
regularity\rq\rq\ classes. Given another regularity structure
$\fThat=(\hat{\cA},\hat{\cT},\hat{\cG})$, one writes $\fT\subset \fThat$ if
$\cA\subset \hat{\cA}$, there exists an injective map $\iota:\cT\ra \hat{\cT}$
such that $\iota(\cT_\alpha)\subset \hat{\cT}_\alpha$ for all $\alpha\in \cA$,
$\iota(\cT)$ is invariant under $\hat{\cG}$, and $j$ given via
$j(\cdot):=\iota^{-1}(\cdot) \iota$ is a surjective group homomorphism from
$\hat{\cG}$ to~$\cG$.
We endow $\R^{d+1}$, using coordinates
$z=(t,x)\in\R\times\R^{d}$, with a \emph{scaling} 
\benn
\fs := (\fs_0,\fs_1,\ldots,\fs_d)\in\N^{d+1}
\eenn
which induces a \emph{scaled degree} for each multiindex $k\in\N_0^{d+1}$ given
by
\benn
|k|_\fs:=\sum_{i=0}^{d} \fs_i k_i
\eenn
and a \emph{scaling map} 
\benn
\cS_\fs^\delta:\R^{d+1}\ra \R^{d+1},\qquad \cS_\fs^\delta(t,x):=
(\delta^{-\fs_0}t,\delta^{-\fs_1}x_1,\ldots,\delta^{-\fs_d}x_{d})\;,
\eenn
as well as a \emph{scaled metric} on $\R^{d+1}$ defined as
\benn
d_\fs((t,x),(s,y)):=|t-s|^{1/\fs_0}+\sum_{i=1}^{d}|x_i-y_i|^{1/\fs_i}\;.
\eenn
It is useful to employ the notations $|\fs|:=\sum_{i=0}^{d} \fs_i$ as well
as $d_\fs((t,x),(s,y))=:\|(t,x)-(s,y)\|_\fs$ although $\|\cdot\|_\fs$ is not
a norm. The associated shifted scalings for a function $\phi$ are defined as
\benn
\bigl(\cS^\delta_{\fs,\bar{z}}\phi\bigr)(z):=\delta^{-|\fs|}
\phi\bigl(\cS^\delta_\fs(z-\bar{z})\bigr)\;.
\eenn
Essentially, one employs scalings in the theory of regularity structures to 
bring the action of certain differential operators on a common invariant 
scale. Here we restrict to the scalings necessary for the operator 
$\partial_t+(-\Delta)^{\rho/2}$, which explains the splitting between the first
time coordinate and the spatial coordinates in the scaling.

\begin{ex}
\label{ex:regpoly}
The \emph{polynomial regularity structure} $(\bar{\cA},\bar{\cT},\bar{\cG})$ on
$\R^{d+1}$ associated with the scaling~$\fs$ is given by $\bar{\cA}=\N_0$,
by the polynomial ring $\bar{\cT}=\R[X_0,X_1,\ldots,X_d]$ 
with a grading induced by the scaled degree 
\benn
\bar{\cT}_n:=\textnormal{span}\{X^k:|k|_\fs=n\}\;, 
\eenn
and $\bar{\cG}$ given by translations. Here we use the multiindex notation
$X^k = X_0^{k_0}\dots X_d^{k_d}$. Indeed, the abstract polynomials have 
a natural structure group with elements $\Gamma_h\in\bar{\cG}$ acting by 
$\Gamma_hX^k=(X+h)^k$ for $h\in\R^{d+1}$ so $\bar{\cG}\simeq \R^d$;
see~\cite[Sec~2.2]{Hairer1}. The requirement~\eqref{eq:lowerorder} then just 
means for the polynomial regularity structure that translating a polynomial 
and then subtracting the original one leaves only lower-degree terms. For
all regularity structures considered here, we always assume that
$\bar{\cT}\subset\cT$.
\end{ex}

In general, whenever $\tau\in\cT_\alpha$, we declare its \emph{homogeneity}
to be given by $|\tau|_\fs = \alpha$. The homogeneity is supposed to reflect a
type of H\"older continuity with respect to the scaled metric $d_\fs$ of the
function or distribution $\tau$ represents. If $\alpha>0$, then a function
$f:\R^{d+1}\to\R$ belongs to $\cC^\alpha_\fs$ if it satisfies a similar
condition on increments as classical H\"older functions, but for the scaled
metric (see~\cite[Def.~2.14]{Hairer1} for a precise definition). If $\alpha<0$,
then a Schwartz distribution $\xi$ is said to belong to $\cC^\alpha_\fs$ if
there exists a constant $C$ such that 
\begin{equation}
 \bigl|\langle \xi, \cS^\delta_{\fs,z}\eta \rangle\bigr| \leqs C\delta^\alpha
\end{equation} 
for any sufficently smooth test function $\eta$ supported in a $d_\fs$-ball of
radius $1$ and any $\delta\in(0,1]$ (see~\cite[Def.~3.7]{Hairer1}). 

There are two (out of several more) important steps we address in this paper 
regarding regularity structures for~\eqref{eq:AC}. Firstly, we have to check 
whether the fractional Laplacian fits into the framework of singular integral 
operators required in the theory of regularity structures~\cite{Hairer1}. Only 
if this is the case, we have a hope of being able to directly apply the
theory. 
This step is quite straightforward. The second step we cover here is to build
the index set $\cA$ and model space $\cT$ and study their dependence upon
parameters. This step is already substantially more involved. 

There are two important further steps in the analysis of the SPDE that we do
not consider here, but for which a general method is given
in~\cite{Hairer1,BrunedHairerZambotti,ChandraHairer16}. The first step is the
definition and analysis of a fixed-point equation, equivalent to a regularised
version of the SPDE, but formulated in a space of modelled distributions
$\cD^\gamma$ (an analogue of the H\"older space on the level of the regularity
structure, cf.~\cite[Def.~3.1]{Hairer1}). The second step is the renormalisation
procedure needed to make
sense of the limit of vanishing regularisation. For both steps, a good
understanding of the model space is essential. Indeed, we observe the following:

\begin{itemize}
\item The negative-homogeneous sector $\bigoplus_{\alpha\in\cA\cap\R_-}
\cT_\alpha$ carries the distribution-valued part of the solution. While in the
case of the standard Laplacian ($\rho=2$) with space-time white noise, this part
only contains the stochastic convolution of heat kernel and noise, for general
fractional Laplacians it may contain many more terms.
 
\item The constants needed to renormalise the equation are determined by
summing over contractions of subtrees in negative-homogeneous elements, as
discussed in~\cite{Hairer4,BrunedHairerZambotti,ChandraHairer16}. We expect that
the number of renormalisation constants diverges as one approaches the
subcriticality boundary.
\end{itemize}

\section{Singular Kernels}
\label{sec:kernels}

To formulate fixed-point equations associated to (S)PDEs in the context of 
regularity structures, we have to consider integration against singular integral 
kernels. Consider a linear differential operator $L$ with constant coefficients 
acting on $u=u(t,x)=u(z)$ for $z\in\R^{d+1}$. Let $G=G(t,x)=G(z)$ denote 
the fundamental solution or kernel, i.e.,
\benn
(LG)(z) = \bm{\delta}(z)\;, 
\eenn 
for $t>0$, where $\bm{\delta}\in \cS(\R^{d+1})$ denotes the delta-distribution.
Frequently, $G$ is a singular kernel. Consider for instance the heat operator
$L:=\partial_t-\sum_{j=1}^d\partial_{x_jx_j}^2$ with associated heat kernel
\benn
G_2(z,\bz) 
= \frac{1}{|4\pi (t-\bar t\,)|^{d/2}}
\exp\left(-\frac{\|x-\bx\|^2}{4(t-\bar t\,)}\right)\;.
\eenn
Here $\bz=(\bt,\bx)\in\R^{d+1}$ and the subscript of $G$, fixed here to $2$, is
used to distinguish the heat kernel $G_2$ from other kernels $G_\rho$ we are
going to consider below, while the subscript will be omitted for the
general theory of singular kernels. Note that $G_2$ is singular at
$z=\bz$. 

\subsection{Hairer's Singular Kernels}

In Hairer's theory of regularity structures, one first aims to decompose
a general kernel $G:\R^{d+1}\times \R^{d+1}\ra \R$ via
\benn
G(z,\bz)=K(z,\bz)+R(z,\bz)\;,
\eenn
where $K$ is the singular part and $R$ is a smooth part. So the key object is
$K$ and we recall the assumptions on $K=K(z,\bz)$ as stated in \cite[Assumption
5.1 / Assumption 5.4]{Hairer1}. A mapping $K:\R^{d+1}\times \R^{d+1}\ra \R$ is
called a \emph{regularizing kernel of order $\beta>0$} if $K$ can be decomposed
as 
\be
\label{eq:kernel_decomp}
K(z,\bz)=\sum_{n\geq 0} K_n(z,\bz)
\ee
and for all $n\geq 0$, $\text{supp}(K_n)\subseteq \{(z,\bz)\in\R^{d+1}\times 
\R^{d+1}:\|z-\bz\|_\fs\leq 2^{-n}\}$, for any $k,l\in \N_0^{d+1}$, there
exists a constant $C$ such that
\begin{align}
\bigl|\txtD_1^k\txtD_2^l K_n(z,\bz)\bigr| &\leqs C
2^{(|\fs|-\beta+|l|_\fs
+|k|_\fs)n}\;, 
\nonumber
%\label{eq:Aker1}
\\
\left|\int_{\R^{d+1}} (z-\bz)^l\txtD_2^k K_n(z,\bz)~\txtd z\right| 
&\leqs C
2^{-\beta n}\;, 
\nonumber
%\label{eq:Aker2}
\\
\left|\int_{\R^{d+1}} (\bz-z)^l\txtD_1^k K_n(z,\bz)~\txtd \bz\right|
&\leqs C 2^{-\beta n}\;, 
\label{eq:Aker3}
\end{align}
hold uniformly over all $n\geq 0$ and all $z,\bz\in\R^d$, and there exists $r>0$
such that 
\be
\label{eq:annihilate}
\int_{\R^{d+1}} K_n(z,\bz)P(\bz)~\txtd \bz = 0
\ee
for every $n\geq 0$, every $z\in\R^d$ and every monomial $P$ of degree
$k^*\in\N_0^{d+1}$ with $|k^*|_\fs\leq r$. We also refer to the condition
\eqref{eq:annihilate} as $r$-order annihilation of polynomials.\medskip

If $G(z,\bz)=G(z-\bz)$ just depends upon a single variable, it is possible
to identify $\beta$-regularizing kernels with point singularities via scaling. 

\begin{lem}[{\cite[Lem. 5.5]{Hairer1}}] 
\label{lem:ker1}
Let $\bar{K}:(\R^{d+1}\setminus\{0\})\ra \R$ be smooth and suppose there
exists $\beta>0$ such that
\be
\label{eq:scale_ker}
\bar{K}(\cS_{\fs}^\delta z)=\delta^{|\fs|-\beta} \bar{K}(z)
\ee
holds for all $z\neq 0$ and for all $\delta \in(0,1]$. Then it is possible to
write $\bar{K}(z)=K(z)+R(z)$
where $R\in \cC^\I(\R^{d+1},\R)$ and $(z,\bz)\mapsto K(z-\bz)$ is a
regularizing kernel of order $\beta$.
\end{lem}

As an example, consider a parabolic scaling $\fs=(2,1,1,\ldots,1)$ and the heat
kernel $\bar{K}=G_2$. Note that $|\fs|=d+2$. Then one can easily check that
\benn
G_2(\cS_{\fs}^\delta z)=\frac{1}{(4\pi
\delta^{-2}t)^{d/2}}\exp\Biggl(-\frac{1}{4t\delta^{-2}}
\sum_{j=1}^d (x_j\delta^{-1})^2\Biggr)=\delta^{d}G_2(z)\;.
\eenn
From \eqref{eq:scale_ker} the condition $|\fs|-\beta=d+2-\beta\stackrel{!}{=}d$
is required, i.e., taking $\beta=2$ means that the heat kernel is a regularizing
kernel of order $2$. In the following, we want to apply the theory of singular 
kernels also to nonlocal diffusion operators described by the fractional Laplacian. 

\subsection{Fractional Singular Integral Operators}

Let $G_\rho(t,x,y)=G_\rho(t,x-y)$ be the transition density of a rotationally
symmetric $\rho$-stable L\'evy process with \emph{L\'evy measure}
$\nu(\txtd x)= \kappa(d,\rho)\|x\|^{-d-\rho}~\txtd x$, where $\kappa(d,\rho)$ is
a normalizing constant depending upon $d\in\N$ and $\rho\in(0,2]$. Then one has
\be
\label{eq:heatfL}
G_\rho(t,x):=\frac{1}{(2\pi)^{d}}\int_{\R^{d}}
\txte^{\txti x^\top \xi}~\txte^{-t\|\xi\|^\rho}~\txtd \xi
\ee
for $x\in\R^d$, $t>0$.
The transition density~\eqref{eq:heatfL} generates a semigroup via
\benn
P_tf(x):=\int_{\R^d}G_\rho(t,x-y)f(y)\,\txtd y\;.
\eenn
The infinitesimal generator of $P_t$ is given by the \emph{fractional
Laplacian}
\benn
-(-\Delta)^{\rho/2}f(x)=\kappa(d,\rho)\lim_{\varepsilon\ra
0^+}\int_{\|y\|>\varepsilon}
\frac{f(x+y)-f(x)}{\|y\|^{d+\rho}}\,\txtd y\;,
\eenn
and we will write $\Delta^{\rho/2}:=-(-\Delta)^{\rho/2}$ to simplify the
notation. One also refers to $G_\rho(t,x-y)$ as the \emph{heat kernel of the
fractional Laplacian}. We mention that there are many different definitions of
the fractional Laplacian~\cite{Kwasnicki}, but the probabilistic one via
\eqref{eq:heatfL} is particularly convenient in our context as it makes the next
result very transparent upon using classical results on L\'evy processes. 

\begin{prop}
\label{prop:FLrho}
Let $\rho\in(0,2]$ and consider the scaling $\fs=(\rho,1,1,\ldots,1)$. Then the
fractional 
heat kernel $G_\rho$ is a regularizing kernel of order $\rho$. 
\end{prop}

\begin{proof}
It is well-known from the theory of L\'evy processes~\cite{Sato} (and can be
seen on~\eqref{eq:heatfL}) that $G_\rho$ satisfies the scaling law
$G_\rho(t,x)=t^{-d/\rho}G_\rho(1,t^{-1/\rho}x)$, so that
\benn
G_\rho(\cS_{\fs}^\delta
z)=G_\rho(\delta^{-\rho}t,\delta^{-1}x)=t^{-d/\rho}\delta^{d}G_\rho(1,t^{-1/\rho
}x)
=\delta^dG_\rho(z)\;.
\eenn
Furthermore, $|\fs|-\rho=\rho+d-\rho=d$ and applying Lemma \ref{lem:ker1}
finishes the proof.
\end{proof}

\begin{rmk}
Strictly speaking, the reconstruction theorem allowing to apply the theory
of regularity structures (cf.~\cite[Thm.~3.10]{Hairer1}) requires $\rho$ to be
a rational number, because the scaling $\fs$ should consist of integers. If
$\rho=p/q\in\Q$, this is not a problem since one can take $\fs=(p,q,\dots,q)$
instead of $\fs=(\rho,1,\dots,1)$ for the scaling. However, it is expected that
the reconstruction theorem also holds for incommensurable scaling vectors. 
Furthermore, the construction of a model space as an independent object does
not require rationality of $\rho$. 
To simplify the notation, in all that follows we will stick to the notation
$\fs=(\rho,1,\dots,1)$ for the scaling, because multiplying $\fs$ by a positive
integer does not affect the results.
\end{rmk}%

\begin{rmk}
It may look natural to try to prove Proposition \ref{prop:FLrho} differently. 
It is known from the theory of fractional Laplace operators, see
e.g.~\cite{ChenKimSong}, that
\benn
\label{eq:scaleFL}
G_\rho(t,x-y)\asymp \min\left(t^{-d/\rho},\frac{t}{\|x-y\|^{d+\rho}}\right)\;.
\eenn
Taking $\fs=(\rho,1,1,\ldots,1)$, it follows that
\benn
\label{eq:ascalefL}
G_\rho(\cS_{\fs}^\delta z)\asymp \delta^{|\fs|-\beta} G_\rho(z)
\eenn
holds for $\beta=\rho$. However, it seems more convenient to work with the exact
scaling property to implement the strategy in the proof of Lemma \ref{lem:ker1} 
from \cite[Lem. 5.5]{Hairer1} verbatim. 
\end{rmk}

\section{Building the Model Space}
\label{sec:canreg}

There is a general procedure to build a regularity structure for an SPDE as
discussed in~\cite[Sec.~8]{Hairer1}. However, the structure itself can be
different for each particular equation. Since we want to consider an entire
family of equations, we have to demonstrate how the regularity structures differ
for the members of the family.  
  
\subsection{Local Subcriticality}
\label{ssec:lsubcrit}

Let $\alpha<0$ denote the smallest upper bound for the H\"older regularity of
the driving noise $\xi$ and recall that $d$ denotes the spatial dimension while
$\rho$ is the regularizing order of the fractional Laplacian. It is known that
for the classical case of space-time white noise and the parabolic scaling, one
has $\alpha =-(d+2)/2$. The generalization for the fractional scaling is as
follows.

\begin{lem}
For the scaling $\fs=(\rho,1,1,\dots,1)$, the smallest upper bound for the
H\"older regularity of space-time white noise $\xi$ is given by 
\begin{equation}
\label{eq:scaling_whitenoise} 
\alpha = - \frac{\rho+d}{2}\;.
\end{equation}
\end{lem}
\begin{proof}
This follows again from a scaling argument. Consider
scaled versions $\xi_\delta$ of $\xi$ defined by 
\[
 \langle\xi_\delta,\eta\rangle =
\langle\xi,\cS^{\delta}_{\fs,0}\eta\rangle
\]
for every test function $\eta$.
The property $\E[\xi(z)\xi(\bz)] = {\bf \delta}(z-\bz)$ of space-time white
noise yields  
\[
\E\bigl[\langle\xi_\delta,\eta\rangle^2 \bigr]
= \int_{\R^{d+1}} (\cS^{\delta}_{\fs,0}\eta(t,x))^2 \,\txtd t\,\txtd x 
= \delta^{-(\rho+d)} \int_{\R^{d+1}} \eta(\bar t,\bar x)^2 \,\txtd \bar
t\,\txtd \bar x
= \cO(\delta^{-(\rho+d)})\;.
\]
The result then follows from the fact that the $p$-th moment of a Gaussian
distribution scales like the power $\frac{p}2$ of its second moment
and the Kolmogorov-type continuity theorem~\cite[Theorem~2.7]{Chandra_Weber_15}.
\end{proof}

We can now check under which algebraic conditions the theory of
regularity structures from~\cite{Hairer1} applies to the family of 
SPDEs~\eqref{eq:AC}. Consider SPDEs of the form
\be
\label{eq:SPDE_sec8}
\partial_t u = Lu +f(u)+\xi\;,
\ee
where $f$ is a polynomial and $L$ is a differential operator inducing a
regularizing kernel of order $\beta$. One defines 
\be
\label{eq:dummy}
F(U,\Xi)=f(U)+\Xi
\ee
where $\Xi$ and $U$ are dummy variables. Each term in the expression
\eqref{eq:dummy} gets assigned a homogeneity reflecting the H\"older class
of the function or distribution it represents. Define $\Xi$ to have 
homogeneity $\alpha_0 = \alpha-\kappa$, where $\kappa>0$ is a fixed
arbitrarily small constant, $U$ to have homogeneity $\alpha_0+\beta$, and 
apply the usual sum rule for exponents of product terms. 
Then~\eqref{eq:SPDE_sec8} is called \emph{locally subcritical} if all terms 
in $f(U)$ have homogeneity strictly greater than $\alpha_0$; 
see~\cite[Sec.~8]{Hairer1}.  

\begin{prop}
\label{prop:genACsub}
For $\fs=(\rho,1,1,\ldots,1)$, the fractional Allen--Cahn 
equation~\eqref{eq:AC} is locally subcritical if and only 
if either (i) $\alpha_0+\rho>0$, or
(ii) $\rho>-\frac{N-1}{N}\alpha_0$, or (iii) $N=0$ hold, 
where $N$ is the degree of $f$ as in~\eqref{eq:poly}.
\end{prop} 

\begin{proof}
First, consider the cases when the polynomial $f$ is nontrivial with $N>0$.
Starting with the case (i), the homogeneity of the term $U^j$ is 
$j(\alpha_0+\rho)$ so if $j_2>j_1$ and $j_1(\alpha_0+\rho)>\alpha_0$ 
then also $j_2(\alpha_0+\rho)>\alpha_0$. Therefore, we have to check the 
local subcriticality condition only for the minimal degree of $f$, which is 
$N=1$. This implies
\be
\alpha_0+\rho>\alpha_0 \qquad \Leftrightarrow \qquad \rho>0
\ee
and $\rho>0$ holds by assumption. The interesting case occurs when the 
noise is irregular and $\alpha_0+\rho\leq0$, which is covered in case (ii). 
As before, the homogeneity of the term $U^j$ is $j(\alpha_0+\rho)$ but now 
if $j_2>j_1$ and $j_2(\alpha_0+\rho)>\alpha_0$ then 
$j_1(\alpha_0+\rho)>\alpha_0$ so we only have to check the term 
of highest degree which yields the requirement
\be
\label{eq:condp1sub}
\beta=\rho>-\frac{N-1}{N}\alpha_0\;,
\ee
so the result claimed in (ii) follows. The case (iii) is trivially 
subcritical as there are no terms to check.
\end{proof}

The last result shows that, as expected, the case $N=0$ is not really 
of interest from the viewpoint of the theory of regularity structures. 
Therefore, we shall assume from now on that $N\in\N$. Of course, 
Proposition~\ref{prop:genACsub} is not yet a practical result as the 
real answer for the fractional Allen--Cahn-type SPDE is hidden in the 
choice of $\xi$.

\begin{thm}
\label{thm:genACsub1}
Let $\xi$ be space-time white noise and $\fs=(\rho,1,1,\ldots,1)$ with 
$\rho\in(0,2]$. The fractional Allen--Cahn equation~\eqref{eq:AC} 
is locally subcritical if and only if 
\be
\label{eq:wnres}
\rho>  d\frac{N-1}{N+1} =: \rhocrit(N,d)\;.
\ee
\end{thm}

\begin{proof}
With~\eqref{eq:scaling_whitenoise}, Condition 
(ii) from Proposition~\ref{prop:genACsub} becomes
\be
\rho>\frac{(N-1)}{N}\frac{(\rho+d)}{2}\qquad \Leftrightarrow \qquad
\rho>d\frac{N-1}{N+1}\;. 
\ee
This is weaker than the condition $\rho>d$ resulting from Condition (i).
Finally, Condition (iii) is ruled out by assumption.
\end{proof}

It is interesting to apply the condition~\eqref{eq:wnres}
to different dimensions to determine which type of nonlinearity
is allowed based on the range of $\rho$.

\begin{cor}
\label{cor:subAC}
Let $\xi$ be space-time white noise, $\fs=(\rho,1,1,\ldots,1)$ and
$\rho\in(0,2]$. 
Then the subcriticality threshold $\rhocrit$ of the fractional Allen--Cahn-type
equation \eqref{eq:AC} belongs to $(0,2)$ in the following cases:
\begin{itemize}
 \item if $d=1$, $f$ can be an infinite series,
 \item if $d=2$, $f$ must be a finite series,
 \item if $d=3$, $N\leq 4$, 
 \item if $d=4$, $N\leq 2$,
 \item if $d=5$, $N\leq 2$,
 \item if $d\geq 6$, $N=1$.
\end{itemize} 
\end{cor} 

The proof is a direct calculation using Theorem~\ref{thm:genACsub1}. We briefly
comment on the result. For $d=1$, we can essentially allow for any analytic
function represented as converging Taylor series. For $d=2$, one observes that
the right-hand side of \eqref{eq:wnres} converges to $2$ if $N\ra +\I$ so
only finitely many terms may appear. For $d=3,4$ one checks that the cases
$N=5,3$ are precisely critical requiring $\rho>2$ while for $d=5$ we obtain for
$N=2,3$ the conditions $\rho>5/3$ and $\rho>5/2$. For all other dimensions, only
linear equations are trivially subcritical.\medskip 

In principle, one could now just apply the \lq\lq Metatheorem 8.4\rq\rq\ of
\cite{Hairer1} to obtain the existence and uniqueness of solutions
to~\eqref{eq:AC} from a suitable fixed-point equation. However, this would not
yield any information on the actual elements of the regularity structure and
these elements are crucial to calculate the renormalized SPDE or to determine a
series expansion of the solution. 

\subsection{Index Set and Model Space}
\label{ssec:modelspace}

We assume that the nonlinearity in the SPDE is given by a
polynomial~\eqref{eq:poly} with degree $N$ and fix the natural fractional
scaling $\fs=(\rho,1,1,\ldots,1)$ for $\rho\in(0,2]$. The model space $\cT_F$
adapted to our class of SPDEs is built by enlarging the model space of the
polynomial structure $\bar{\cT}$ by adding symbols and taking into account the
regularity of the noise and the nonlinearity. To each symbol one assigns a
homogeneity $|\cdot|_\fs$, e.g., one sets $|X^k|_\fs:=|k|_\fs$. The noise is
represented by $\Xi$ with homogeneity $|\Xi|_\fs=\alpha_0$. Furthermore, let
$\cI_\rho$ be an abstract integration operator, which increases homogeneity
by $\rho$ by definition, i.e., 
\be
|\cI_\rho(\cdot)|_\fs = |\cdot|_\fs + \rho\;.
\ee
Define a set $\cF$ by declaring $\{{\bf 1},X_i,\Xi\}\subset \cF$, where
${\bf 1}$ is a neutral element for a product to be considered below with $|{\bf
1}|_\fs=0$. Then, postulate that if $\tau,\bar{\tau}\in \cF$ then $\tau
\bar{\tau}\in\cF$ and $\cI_\rho(\tau)\in\cF$. Note that $\tau \bar{\tau}$ and
$\cI_\rho(\tau)$ are then new formal symbols with the natural conventions
understood, e.g., ${\bf 1}(\cdot)=(\cdot)$ and $X_iX_i=X_i^2$. The set $\cF$
contains infinitely many symbols, so that just defining $\cT_\gamma$ by
collecting elements of homogeneity $\gamma$ does not work.\medskip 

For locally subcritical cases of~\eqref{eq:AC}, there exists a recursive
procedure to build a regularity structure containing only finitely many
negatively homogeneous elements by constructing a suitable subset $\cF_F$
depending upon the nonlinearity $f$
via~\eqref{eq:dummy}~\cite[Section~8.1]{Hairer1}. In particular, let
\be
\mathfrak{M}_F:=\{\Xi,U^n:1\leq n\leq N\}
\ee 
i.e., monomials in $\Xi$ and $U$, where $\Xi$ only appears to the power one and
the powers of $U$ are bounded by the polynomial degree of the nonlinearity. If
$A,B\subset \cF$ let 
\be
AB:=\{\tau\bar\tau:\tau\in A,\bar\tau\in B\}\;.
\ee
Set $\cW_0=\{\}=\cU_0$ and recursively define
\begin{align}
\cW_m &= \cW_{m-1}\cup \bigcup_{\cQ\in \mathfrak{M}_F}\cQ(\cU_{m-1},\Xi)\;,
\label{eq:recurse1}\\
\cU_m &= \{X^k\}\cup \{\cI_\rho(\tau):\tau\in \cW_m\}\;, \label{eq:recurse2}
\end{align}
where $k$ runs over all possible multiindices. The notation $\cQ(\cU_{m-1},\Xi)$
also implies that we replace each occurrence of $U$ in a monomial by some
expression from $\cU_{m-1}$. Essentially this recursive construction restricts
the regularity structure to only those symbols necessary for a fixed-point
procedure. If one defines
\be
\label{eq:recurse3}
\cF_F:=\bigcup_{m\geq 0} (\cW_m\cup\cU_m)\;,
\qquad \cU_F:=\bigcup_{m\geq 0}\cU_m\;,
\ee
then $\cF_F$ collects all symbols necessary to represent the equation and
$\cU_F$ all symbols to represent the solution. A very fundamental result about
the construction is that we can now define a regularity structure with suitable
finiteness properties.

\begin{thm}[{\cite[Lem.~8.10]{Hairer1}}] Suppose $\alpha_0<0$. Then the set 
$\{\tau\in\cF_F:|\tau|_\fs \leq \gamma\}$ is finite for every $\gamma\in\R$
if and only if the SPDE is locally subcritical.
\end{thm}

Therefore, Corollary~\ref{cor:subAC} gives a precise criterion for when we
can expect to be able to define a suitable regularity structure via the key
definition
\be
\cA_F:=\{|\tau|_\fs:\tau\in\cF_F\}\;,
\ee
so that $\cT_{F,\gamma}$ is the set of formal linear combinations of elements in
$\{\tau\in\cF_{F}:|\tau|_\fs=\gamma\}$. 
This constructs $\cA_F,\cT_F$ and we postpone the concrete construction 
and analysis of the group $\cG_F$ to future work (a general abstract 
construction of $\cG_F$ is given in~\cite[Section~8.1]{Hairer1}). We are 
faced with the interesting
question of the actual size of $\cT_F$ for different values of $\rho$. For
space-time white noise, $\rho=2$ and a cubic polynomial
\be
f(u)=u-u^3\;,
\ee
it is well understood how $\cT_F$ is given; see~\cite{HairerWeber} 
or~\cite[Table~1]{BerglundKuehn}. However, viewing the problem as a
three-parameter
family with $\rho\in(0,2]$, $N\in\N$ and $d\in\N$ is not trivial. Define
\be
h_F(N,d,\rho):=|\{\gamma:\exists\tau\in \cF_F \text{ with
}|\tau|_\fs=\gamma<0\}|\;,
\qquad h_F:\N\times \N\times (0,2]\ra \N\;,
\ee
so $h_F$ is the counting map for the number of different negative homogeneities
in the regularity structure. These are the elements of interest as those 
elements make the representation different from classical function
representations
by elements of non-negative homogeneity. The 
homogeneity counting map $h_F$ is smaller or equal to the actual element
counting map
\be
c_F(N,d,\rho):=|\{\tau\in \cF_F:|\tau|_\fs<0\}|\;,
\qquad c_F:\N\times \N\times (0,2]\ra \N\;,
\ee
i.e., $c_F\geq h_F$ and $c_F>h_F$ does usually occur as shown in the next
example. 

\begin{ex}
Let $d=2$, $N=3$ and $\rho=2$. Then for some arbitrarily small $\kappa>0$,
one finds
\be
|\cI_2(\cI_2(\Xi)^2)\cI_2(\Xi)^2|_\fs
= -4\kappa = |\cI_2(\cI_2(\Xi)^3)\cI_2(\Xi)|_\fs\;,
\ee
showing that $c_F(3,2,2) \geqs h_F(3,2,2)+1$.
\end{ex}

\subsection{Counting Homogeneities}
\label{ssec:homogeneities} 

A first step towards finding bounds on $c_F$ and $h_F$ is to determine
which is the element of smallest homogeneity in $\cF_F$.

\begin{lem}
\label{lem:maxneg}
Suppose the SPDE~\eqref{eq:AC} is locally subcritical for space-time white
noise. Then the elements of smallest homogeneity of $\cF_F$ are $\Xi$
and
$\cI_\rho(\Xi)^N$ with 
\be
|\Xi|_\fs=-\frac{\rho+d}{2}-\kappa,\qquad 
|\cI_\rho(\Xi)^N|_\fs=\frac{N}{2}(\rho-d)-\kappa N
\ee  
and $|\Xi|_\fs<|\cI_\rho(\Xi)^N|_\fs$\;.
\end{lem}

\begin{proof}
We prove the last statement first. We have, using subcriticality,
\be
\label{eq:step1}
-\rho<d\frac{1-N}{1+N}\quad \Leftrightarrow \quad
-\rho-d<\rho N-dN\;,
\ee
so the result follows upon selecting $\kappa$ sufficiently small. The first
statement about minimality now follows essentially by induction. More precisely,
there can be at most $N$ terms in each new symbol assembled from previous
symbols via the recursion~\eqref{eq:recurse1}--\eqref{eq:recurse2}. To minimize
the homogeneity, one may not include any terms involving polynomials, and one
must maximize the negative homogeneity contributions. 
If $\tau$ is the symbol with smallest homogeneity among symbols before applying
$U^N$, then $\tau^N$ minimizes homogeneity if one excludes $\Xi$. The
calculation~\eqref{eq:step1} shows that homogeneity increases from the first to
the second step of the recursion~\eqref{eq:recurse1}--\eqref{eq:recurse2}. This
step can be taken as the base step for induction on the level $k_r$ of the
recursion. Given some element $\tau$ with
\be
|\cI_\rho(\cI_\rho(\tau)^N)|_{\fs}=N(|\tau|_\fs +\rho) {}+ \rho
\ee 
we must prove that
\be
\label{eq:toprove1}
N(|\tau|_\fs +\rho)>|\tau|_\fs\;.
\ee
Subcriticality and the induction assumption $|\tau|_\fs\geq -(\rho+d)/2$
easily imply~\eqref{eq:toprove1}, and this means the element with smallest
homogeneity at step $k_r+1$ that gets adjoined to $\cF_F$ has bigger homogeneity
than $\cI_\rho(\tau)$. The result follows.  
\end{proof}

The last result essentially shows that the type of recursive procedure which is
used to construct regularity structures for additive noise SPDEs with polynomial
nonlinearities also does yield well-defined elements of smallest homogeneity. 

One may hope that considering space-time white noise, which imposes the more 
stringent restriction~\eqref{eq:wnres}, simplifies the combinatorics enough.
The next result shows that the lower bound provided by the homogeneity
counting map $h_F$ could be quite large for many regularity structures even 
without the free choice of $\Xi$ (resp.~$\alpha_0$). 

\begin{prop}
\label{prop:Dio} 
Consider space-time-white noise. Let $\rho=p/q\in\Q$ and let
$h_{\textnormal{Dio}} =h_{\textnormal{Dio}}(N,d,\rho)$ denote the number of
solutions to the system of constrained Diophantine equations
\be
{\bf A}{\bf c}={\bf b}, \qquad \text{with }{\bf A}\in \Z^{3\times 6},\quad {\bf
b}\in \Z^3\;,\quad 
{\bf c}=(c_1,c_2,c_3,c_4,c_5,c_6)^\top,\quad c_j\in\N_0
\ee
where the matrix ${\bf A}$ and integer vector ${\bf b}$ are computable as
\be
{\bf A}=\begin{pmatrix}
  p & 2q & -dq & 1 & 0  & 0\\
  p & 2q & -dq & 0 & -1 & 0\\
  -1 & 0 & 1 & 0 & 0 & 1\\
  \end{pmatrix}\;,\qquad 
  {\bf b}=\begin{pmatrix}
  dq-p \\
  -(N-1)(dq-p)\\
  0\\
  \end{pmatrix}\;.
\ee
Then $h_F=h_{\textnormal{Dio}}+1$.
\end{prop}

\begin{proof}
By Lemma~\ref{lem:maxneg}, we may restrict to counting homogeneities of elements
$\tau$ with $|\tau|_\fs \geq N(\alpha_0+\rho)$ if we count $\Xi$ separately
which explains the term $+1$ in the claim $h_{\textnormal{Dio}}+1=h_{F}$. The 
remaining homogeneities can be counted by decomposing the recursion steps and
noting that 
\be
|\cI_\rho(\Xi)|_\fs=\frac{\rho}{2}-\frac{d}{2}-\kappa\;.
\ee
Furthermore, $|X_0|_\fs=\rho$, $|X_j|_\fs=1$ for $j\geq 1$, and suitable power
combinations may appear in possible homogeneities. This yields the problem to
find all $c_1,c_2,c_3\in \Z$ such that
\be
\label{eq:basecountprob}
N(\alpha_0+\rho)\leq \frac{\rho}{2}c_1+c_2-c_3\frac{d}{2}\leq 0
\ee
under the constraints $c_1\geq c_3\geq 1$ and $c_2\geq 0$. 
Re-writing~\eqref{eq:basecountprob} as two separate inequalities and using
$\rho=p/q\in\Q$
yields
\benn
pc_1+2qc_2-dqc_3\leq 0\;,\quad pc_1+2qc_2-dqc_3\geq Np-Ndq\;,
\eenn
as well as 
\benn
c_1\geq 1,\quad c_2\geq 0\;,\quad c_3\geq 1\;,\quad c_1-c_3\geq 0\;. 
\eenn
Introducing slack variables $c_4,c_5,c_6$, we get 
\benn
pc_1+2qc_2-dqc_3+c_4= 0\;,\quad pc_1+2qc_2-dqc_3= c_5+Np-Ndq\;, \quad
c_1-c_3-c_6=
0\;.
\eenn
with the remaining constraints unchanged. Shifting $c_1$ and $c_3$ via 
$\tilde{c}_1:=c_1-1$, $\tilde{c}_3:=c_3-1$, re-arranging and dropping 
the tildes yields the result.
\end{proof}

The main insight provided by Proposition~\ref{prop:Dio} is not the precise form
of the equations but the type of combinatorial problem one has to solve. 
In fact, the result already anticipates that classical combinatorial tools, 
e.g.\ using the method of stars-and-bars, are going to be relevant. Furthermore,
the result shows that we cannot expect a closed-form solution for all parameters. 
Hence, we are going to examine the asymptotic behaviour of the homogeneity
counting map $h_F$ as $\rho$ approaches the critical value $\rhocrit$ from
above. To this end, it is useful to introduce a geometric approach. 
Any element $\tau\in\cF_F$ contains a certain number $p(\tau)$ of occurrences of
$\Xi$, a number $q(\tau)$ of occurrences of $\cI_\rho$, and monomials
of total exponent $k\in\N_0^{d+1}$. We will say that $\tau$ is of type
$(p,q,k)$. Its homogeneity is then given by 
\begin{equation}
\label{eq:tau_homogeneity} 
 |\tau|_\fraks = p\alpha_0 + q\rho + |k|_\fraks\;,
\end{equation}
where we recall that 
\begin{equation}
\label{eq:alpha0} 
 \alpha_0 = - \frac{\rho+d}{2} - \kappa\;.
\end{equation}
As a first step, let us consider only elements such that $k=0$. Each
$\tau\in\cF_F$ of this type can be represented by the point $(p,q)\in\N_0^2$. If
for a given $\cU\subset\cF_F$, we let $D_0(\cU)\subset\N_0^2$ be the set of
indices $(p,q)$ of the elements of $\cU$, we are looking for
\begin{equation}
 h_F^0(N,d,\rho) = \bigl|\{(p,q)\in D_0(\cF_F) \colon
 p\alpha_0 + q\rho < 0\}\bigr|\;.
\end{equation}
Obviously, $h_F^0 \leqs h_F \leqs c_F$. 

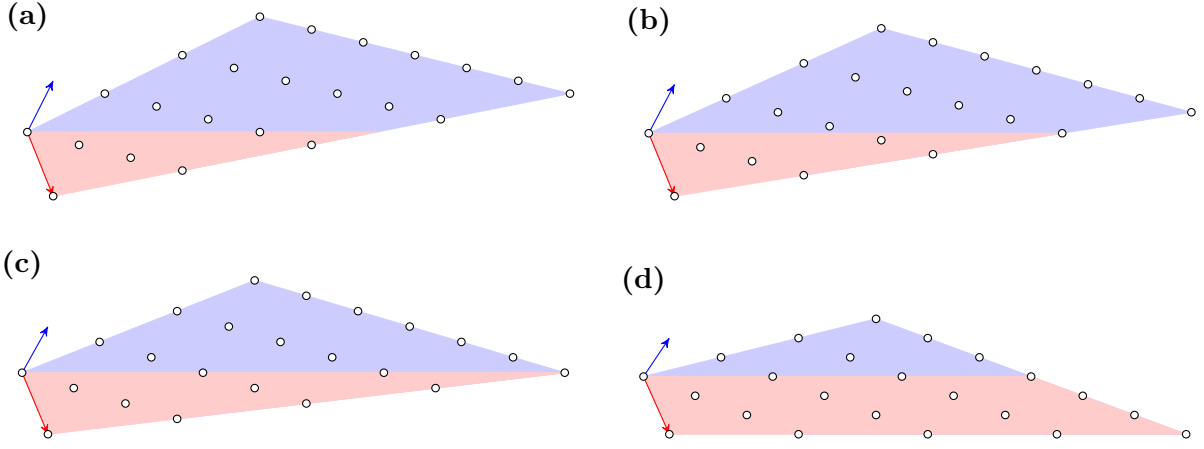
\begin{figure}
\begin{center}
\begin{tikzpicture}[>=stealth',main
node/.style={circle,inner sep=0.035cm,fill=white,draw},scale=0.34,
x={(1,{-2.5})}, y={({1},{4.5})}]
% x = (1,alpha0), y=(1,rho-alpha0)

\draw[thin,blue!20,fill=blue!20] (0,0) -- (3,6) -- (9,12) -- (3,3) -- cycle;

\draw[thin,red!20,fill=red!20] (0,0) -- (6,7.5) -- (1,0) -- cycle;

\draw[red,->] (0,0) -- (1,0);
\draw[blue,->] (0,0) -- (0,1);

\foreach \ij in {(1,0), (0,0), (1,1), (1,2), (2,2), (2,3), (2,4), 
(3,3), (3,4), (3,5), (3,6), (4,5), (4,6), (4,7),
(5,6), (5,7), (5,8), (6,8), (6,9), (7,9), (7,10), (8,11), (9,12)}
{
   \node[main node] at \ij {};
}

\node[] at (-1,1) {{\bf (a)}};

\end{tikzpicture}
\hspace{3mm}
\begin{tikzpicture}[>=stealth',main
node/.style={circle,inner sep=0.035cm,fill=white,draw},scale=0.34,
x={(1,{-2.454546})}, y={({1},{4.363637})}]
%x={(1,{-2.40909})}, y={({1},{4.31818})}]
% x = (1,alpha0), y=(1,rho-alpha0)

\draw[thin,blue!20,fill=blue!20] (0,0) -- (3,6) -- (9,12) -- (3,3) -- cycle;

\draw[thin,red!20,fill=red!20] (0,0) -- (7,9) -- (1,0) -- cycle;

\draw[red,->] (0,0) -- (1,0);
\draw[blue,->] (0,0) -- (0,1);

\foreach \ij in {(1,0), (0,0), (1,1), (1,2), (2,2), (2,3), (2,4), 
(3,3), (3,4), (3,5), (3,6), (4,5), (4,6), (4,7),
(5,6), (5,7), (5,8), (6,8), (6,9), (7,9), (7,10), (8,11), (9,12)}
{
   \node[main node] at \ij {};
}

\node[] at (-1,1) {{\bf (b)}};

\end{tikzpicture}

\vspace{5mm}

\begin{tikzpicture}[>=stealth',main
node/.style={circle,inner sep=0.035cm,fill=white,draw},scale=0.34,
x={(1,{-2.4})}, y={({1},{4.2})}]
%x={(1,{-2.5})}, y={({1},{4.5})}]
% x = (1,alpha0), y=(1,rho-alpha0)

\draw[thin,blue!20,fill=blue!20] (0,0) -- (3,6) -- (9,12) -- (3,3) -- cycle;

\draw[thin,red!20,fill=red!20] (0,0) -- (9,12) -- (1,0) -- cycle;

\draw[red,->] (0,0) -- (1,0);
\draw[blue,->] (0,0) -- (0,1);

\foreach \ij in {(1,0), (0,0), (1,1), (1,2), (2,2), (2,3), (2,4), 
(3,3), (3,4), (3,5), (3,6), (4,5), (4,6), (4,7),
(5,6), (5,7), (5,8), (6,8), (6,9), (7,9), (7,10), (8,11), (9,12)}
{
   \node[main node] at \ij {};
}

\node[] at (-1,1) {{\bf (c)}};

\end{tikzpicture}
\hspace{3mm}
\begin{tikzpicture}[>=stealth',main
node/.style={circle,inner sep=0.035cm,fill=white,draw},scale=0.34,
x={(1,{-2.25})}, y={({1},{3.75})}]
% x = (1,alpha0), y=(1,rho-alpha0)

\draw[thin,blue!20,fill=blue!20] (0,0) -- (3,6) -- (9,12) -- (3,3) -- cycle;

\draw[thin,red!20,fill=red!20] (0,0) -- (6,9) -- (9,12) -- (1,0) -- cycle;

\draw[red,->] (0,0) -- (1,0);
\draw[blue,->] (0,0) -- (0,1);

\foreach \ij in {(1,0), (0,0), (1,1), (1,2), (2,2), (2,3), (2,4), 
(3,3), (3,4), (3,5), (3,6), (4,5), (4,6), (4,7),
(5,6), (5,7), (5,8), (6,8), (6,9), (7,9), (7,10), (8,11), (9,12)}
{
   \node[main node] at \ij {};
}

\node[] at (-1,1) {{\bf (d)}};

\end{tikzpicture}
\vspace{-2mm}
\end{center}
\caption[]{The set $D_0(\cW_3)$ of lattice points $(p,q)$ for elements of
$\cW_3$ with trivial polynomial part, where $p$ is the number of instances of
$\Xi$ and $q$ is the number of instances of $\cI_\rho$. The basis vectors
$(1,0)$ (red) and $(0,1)$ (blue) have been rotated in such a way that the
vertical coordinate gives the homogeneity $p\alpha_0+q\rho$. Parameter values
are $N=d=3$, with {\bf (a)} $\rho=2$, {\bf (b)} $\rho=\frac{21}{11}$, {\bf (c)} 
$\rho=\frac95$ and {\bf (d)} $\rho=\rhocrit(3,3)=\frac32$. The element becoming
negative-homogeneous for $\rho=\frac{21}{11}$ is $(\cI_\rho(\cI_\rho(\Xi)^3))^2
\cI_\rho(\Xi)$ (which is of type $(p,q,k)=(7,9,0)$).}
\label{fig_lattice}
\end{figure}

\figref{fig_lattice} shows the set $D_0(\cW_3)$ for $N=d=3$ and different
values of $\rho$. Note that it is given by the set of lattice points inside
a quadrilateral, and that as $\rho$ decreases, one side of the quadrilateral
becomes aligned with the line of zero homogeneity.

\begin{prop}
For any $N\geqs 2$ one has
\begin{equation}
\label{eq:D0_WF} 
 D_0(\cF_F) = 
 D_0(\cW_\I) = \{(0,0)\} \cup
 \biggl\{(p,q)\in\N\times\N_0 \colon 1 \leqs p \leqs 1 + \frac{N-1}{N}q
\biggr\}\;.
\end{equation}
\end{prop}
\begin{proof}
The first steps of the iterative
construction~\eqref{eq:recurse1}--\eqref{eq:recurse2} give 
\begin{align*}
\cW_1 &= \bigl\{\Xi\bigr\} \;, \\
\cU_1 &= \bigl\{X^k\bigr\} \cup \bigl\{\cI_\rho(\Xi)\bigr\}\;, \\
\cW_2 &= \bigl\{\Xi\bigr\} \cup \cU_1 \cup \dots \cup \cU_1^N\;.
\end{align*}
The only elements without polynomial part in $\cW_2$ are $\Xi, \cI_\rho(\Xi),
\dots, (\cI_\rho(\Xi))^N$, showing that 
\[
 D_0(\cW_2) 
 = \left\{(0,0), (1,0), (1,1), (2,2), \dots, (N,N)\right\}\;.
\]
The set $\cU_2$ is obtained by applying $\cI_\rho$ to $\cW_2$ and adding
polynomials, so that 
\[
 D_0(\cU_2) 
 = \left\{(0,0), (1,1), (1,2), (2,3), \dots, (N,N+1)\right\}\;.
\]
The point $(0,1)$ has been removed because by definition of the integration
operator, $\cI_\rho({\bf 1})=0$. 
The central observation when constructing $\cW_{m+1}$ from $\cU_m$ is that 
\begin{itemize}
\item 	$D_0(\cU_m^j)$ contains all points $(jp,jq)$ with $(p,q)\in
D_0(\cU_m)$; 
\item 	due to cross terms, $D_0(\cU_m^j)$ also contains all lattice points
in the convex envelope of the above points. 
\end{itemize}
We claim that for any $m\geqs3$, 
\begin{itemize}
\item 	$D_0(\cW_m)$ contains all lattice points in the triangle with vertices 
\[
%\label{eq:D0_Wm} 
 (0,0)\;, \quad
 (1,0) \quad \text{and} \quad
 \Bigl(N^{m-1}, \frac{N^m-N}{N-1}\Bigr)\;; 
\]
\item 	$D_0(\cW_m)$ contains the point $(1,m-1)$;
\item 	all points in $D_0(\cW_m)$ satisfy $q\geqs \frac{N}{N-1}(p-1)$.
\end{itemize}
The base case $m=3$ follows easily using the above remarks when constructing
the $D_0(\cU_2^j)$. Indeed, they show that $D_0(\cW_3)$ contains
all lattice points in the quadrilateral with vertices $(0,0)$, $(1,0)$,
$(N^2,N^2+N)$ and $(N,2N)$, cf.~\figref{fig_lattice}. 

The induction step proceeds as follows. First, $D_0(\cU_m)$ is obtained by
shifting $D_0(\cW_m)$ by one step in the $q$-direction, removing the point
$(0,1)$ and adding $(0,0)$. In particular, $D_0(\cU_m)$ contains all the lattice
points in the triangle with vertices $(0,0)$, $(1,1)$ and
$(N^{m-1},\frac{N^m-1}{N-1})$. Next, we see that $D_0(\cW_{m+1})$ contains all
lattice points in the image of this triangle under scaling by a factor $N$, as
well as $(1,0)$, and these points form exactly the triangle required at
induction step $m+1$. Furthermore, $D_0(\cW_{m+1})$ contains $(1,m)$ because
$D_0(\cU_m)$ does, and the inequality for $q$ is satisfied by all points in
$D_0(\cU_m^j)$ with $1\leqs j\leqs N$.

The conclusion follows by taking the limit $m\to\infty$, using again a
convexity argument.
\end{proof}

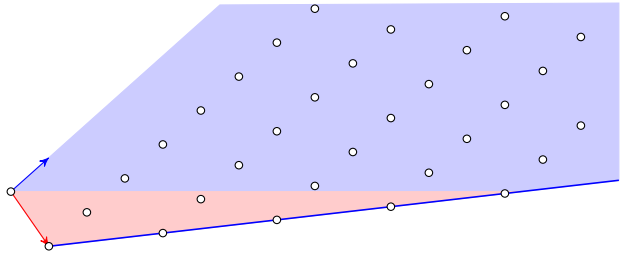
\begin{figure}
\begin{center}
\begin{tikzpicture}[>=stealth',main
node/.style={circle,inner sep=0.035cm,fill=white,draw},scale=0.5,
x={(1,{-1.45})}, y={({1},{2.35})}]
%x={(1,{-2.5})}, y={({1},{4.5})}]
% x = (1,alpha0), y=(1,rho-alpha0)

%\draw[thin,blue!20,fill=blue!20] (0,0) -- (4,12) -- (6,10) -- (1,0) -- cycle;
\draw[thin,blue!20,fill=blue!20] (0,0) -- (0,5.5) -- (4,12) -- (6,10) -- (1,0)
-- cycle;

\draw[thin,red!20,fill=red!20] (0,0) -- (5.14286,8.2857) -- (1,0) -- cycle;

%\draw[semithick,blue] (0,0) -- (4,12);
\draw[semithick,blue] (1,0) -- (6,10);

\draw[red,->] (0,0) -- (1,0);
\draw[blue,->] (0,0) -- (0,1);

\foreach \ij in {(0,0), 
(1,0), (1,1), (1,2), (2,2), (1,3), (1,4), (1,5), (1,6), (1,7), 
(2,3), (2,4), (2,5), (2,6), (2,7), (2,8), 
(3,4), (3,5), (3,6), (3,7), (3,8), (3,9), (3,10), 
(4,6), (4,7), (4,8), (4,9), (4,10), (4,11), 
(5,8), (5,9), (5,10)}
{
   \node[main node] at \ij {};
}

\end{tikzpicture}
\vspace{-2mm}
\end{center}
\caption[]{The set $D_0(\cF_F)=D_0(\cW_\I)$ for $d=N=2$ and $\rho=0.9$.
There are exactly $7$ pairs $(p,q)$ yielding negative-homogeneous elements,
namely $(1,0)$, $(1,1)$, $(2,2)$, $(2,3)$, $(3,4)$, $(4,6)$ and $(5,8)$. See
also Figure~\ref{fig:02} for the associated elements of the model space.}
\label{fig_N=2}
\end{figure}

The set $D_0(\cF_F)$ is the intersection of a truncated cone with the integer
lattice (\figref{fig_N=2}). 
If $\alpha_0/\rho$ is irrational, the number $h^0_F(N,d,\rho)$ of elements of
negative homogeneity is equal to the number of lattice points in this truncated
cone that lie below the line of slope $-\alpha_0/\rho$ (\figref{fig_D0_D}). If
$\alpha_0/\rho$ is rational, many elements will share the same homogeneity, but
this only occurs on a parameter set of measure zero. Furthermore, as pointed out
in~\cite{Hairer1}, one may always slightly shift $\alpha_0$ to avoid such \lq\lq
resonances\rq\rq.

The number $h^0_F(N,d,\rho)$ diverges as the slope  $-\alpha_0/\rho$ approaches
$N/(N-1)$ (see \figref{fig_D0_D}). This corresponds exactly to $\rho$
approaching the subcriticality threshold 
\begin{equation}
 \rhocrit(N,d) = d \frac{N-1}{N+1}\;.
\end{equation}
We can compute the way in which $h^0_F(N,d,\rho)$ diverges by estimating the
number of lattice points in the part of $D_0(\cF_F)$ below the line
$q=(-\alpha_0/\rho)p$. 

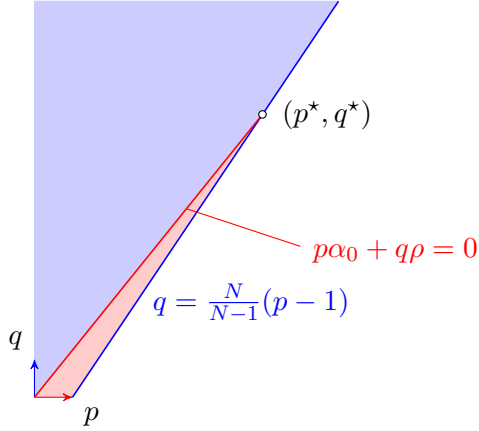
\begin{figure}
\begin{center}
\begin{tikzpicture}[>=stealth',main
node/.style={circle,inner sep=0.035cm,fill=white,draw},scale=0.5]
% x = (1,alpha0), y=(1,rho-alpha0)

\draw[thin,blue!20,fill=blue!20] (0,0) -- (1,0) -- (8,10.5) --
(0,10.5) -- cycle;

\draw[thin,red!20,fill=red!20] (0,0) -- (1,0) -- (6,7.5)  -- cycle;

\draw[red,->] (0,0) -- (1,0);
\draw[blue,->] (0,0) -- (0,1);

\draw[semithick, blue] (1,0) -- (8,10.5);
%\draw[semithick, blue] (0,0) -- (5,12.5);
\draw[semithick, red] (0,0) -- (6,7.5);

\draw[thin, red] (4,5) -- (7,4);

\node[] at (1.5,-0.5) {$p$};
\node[] at (-0.5,1.5) {$q$};

\node[red] at (9.5,3.9) {$p\alpha_0 + q\rho=0$};
\node[blue] at (5.7,2.5) {$q = \frac{N}{N-1}(p-1)$};
%\node[blue] at (1.5,9) {$q = \frac{2N-1}{N-1}p$};

\node[main node] at (6,7.5) {};
\node[] at (7.7,7.5) {$(p^\star,q^\star)$};

\end{tikzpicture}
\vspace{-5mm}
\end{center}
\caption[]{The set $D_0(\cF_F)$ for $d=N=3$ and $\rho=2$ (lattice points in
the blue and red regions). The red triangular region corresponds to negative
homogeneous elements. It contains at most $\lfloor q^\star \rfloor +1$ lattice
points.}
\label{fig_D0_D}
\end{figure}

\begin{prop}
\label{prop:h0F} 
For any $\rho > \rhocrit$, if $\kappa$ is sufficiently small then 
\begin{equation}
\label{eq:area} 
 \frac{\rho+d}{N+1} \bigl(\rho - \rhocrit\bigr)^{-1}
 \leqs h^0_F(N,d,\rho) 
 \leqs 1 + \frac{(\rho+d)N}{N+1} \bigl(\rho - \rhocrit\bigr)^{-1}\;.
\end{equation} 
\end{prop}
\begin{proof}
The line $q=(-\alpha_0/\rho)p$ intersects the truncated cone $D_0(\cF_F)$ at a
point $(p^\star,q^\star)$ with coordinates
\begin{align}
\nonumber
 p^\star &= \frac{2\rho N}{\rho(N+1) - d(N-1) -2\kappa(N-1)}
 = \frac{2\rho N}{(\rho - \rhocrit)(N+1) -2\kappa(N-1)}\;, \\
 q^\star &= \frac{(\rho+d+2\kappa) N}{\rho(N+1) - d(N-1) -2\kappa(N-1)}
 = \frac{(\rho+d+2\kappa)N}{(\rho - \rhocrit)(N+1) -2\kappa(N-1)}\;. 
 \label{eq:pstar-qstar} 
\end{align}
The region containing points with negative homogeneity is a triangle as 
shown in \figref{fig_D0_D}. For any $0\leqs q \leqs q^\star$, this triangle
contains all points $(p,q)$ such that 
\begin{equation}
\label{eq:condition_pq} 
 \frac{2\rho}{\rho+d+2\kappa}q < p \leqs \frac{N-1}{N} q + 1\;. 
\end{equation} 
The condition $\rho > \rhocrit$ implies that $p$ lies in an interval of width
strictly less than $1$ if $\kappa$ is small enough, except for $q=0$, where the
width is exactly $1$. If $q=0$, however, only the case $p=1$ corresponds to a
negative homogeneity. Therefore, for each $q$ there is at most one lattice point
in the triangle. On the other hand, the triangle contains at least all points
$(\frac{N-1}{N} q + 1,q)$ for which $q \leqs q^\star$ is a multiple of $N$. It
follows that 
\[
 1 + \biggl\lfloor\frac{q^\star}{N}\biggr\rfloor \leqs h^0_F(N,d,\rho) 
 \leqs 1 + \lfloor q^\star\rfloor 
\]
which implies the bounds~\eqref{eq:area}, taking $\kappa$ sufficiently small. 
\end{proof}

We expect that as $\rho\searrow\rhocrit$, $h^0_F(N,d,\rho)$ will be closer
to the
upper bound in~\eqref{eq:area}, since the red triangle approaches a strip of
constant width $1$. 

\begin{rmk}
\label{remark:pq} 
The proof shows that for $q \leqs q^\star$, there is at most one value
of $p$ such that the corresponding element has a negative homogeneity, given by 
\begin{equation}
\label{eq:pq} 
 p = 1 + \biggl\lfloor\frac{N-1}{N}q\biggr\rfloor\;. 
\end{equation} 
This means that for a given nonzero number of integration operators $\cI_\rho$,
there is at most one choice for the number of symbols $\Xi$ yielding a
negative-homogeneous symbol.  
\end{rmk}

\begin{thm}
\label{thm:hF} 
For any $\rho > \rhocrit$, the homogeneity counting map satisfies
\begin{equation}
\label{eq:hF} 
 \frac{\rho+d}{N+1}\bigl(\rho - \rhocrit\bigr)^{-1}  
 \leqs h_F(N,d,\rho) \leqs 
 1 + \frac{(\rho+d)d N}{N+1}\bigl(\rho - \rhocrit\bigr)^{-1}
\end{equation}
if $\kappa$ is small enough. 
\end{thm}
\begin{proof}
The lower bound follows directly from Proposition~\ref{prop:h0F}. To obtain an
upper bound, we have to control the number 
\[
 m(p,q) = \bigl| \left\{ (r,s)\in\N_0^2 \colon r\rho +  s < \theta(p,q) :=
(-\alpha_0) p - \rho q \right\}\bigr|
\]
of possible homogeneities obtained by adding polynomial terms to an element of
$D_0(\cF_F)$. Remark~\ref{remark:pq} shows that $p \leqs 1+\frac{N-1}{N}q$.
Using the definition of $\alpha_0$ and $\rho>\rho_c$, it follows that 
\[
 \theta(p,q) \leqs \frac{\rho+d}{2} + \kappa
\]
uniformly in $(p,q) \in D_0(\cF_F)$. Using again that $\rho>\rhocrit$ and 
$r\rho+s<\theta(p,q)$, we see that this imposes 
\[
 r < \frac12 \biggl( 1 + \frac{d}{\rho} \biggr) + \kappa 
 < \frac{N}{N-1} + \kappa\;.
\]
If $N\geqs3$ then this enforces $r\in\{0,1\}$ if $\kappa$ is small enough.
If $r=0$ then one must have $s \leqs (\rho+d)/2$ while if $r=1$ then
$r\leqs(d-\rho)/2$. Therefore 
\[
m(p,q) \leqs d\;.
\]
If $N=2$ then $r=2$ is also allowed, yielding $m(p,q)\leqs 3(d-\rho)/2$.
However, the condition $\rho>\rhocrit=d/3$ then implies $m(p,q)\leqs 3d/4\leqs
d$. Since the only element appearing for $q=0$ is $\Xi$, which is never
multiplied by a polynomial, we can bound $h_F$ by $(\rho+d)(h^0_F-1)+1$, which
yields the result.
\end{proof}

\begin{rmk}
The upper bound in~\eqref{eq:hF} is not sharp. In particular, it is possible to
obtain a sharper bound on $\theta(p,q)$ proportional to $q^\star-q$.
Furthermore, the uniform upper bound on $m(p,q)$ overestimates its actual
value. However, this will only affect the numerical constant in front of $(\rho
- \rhocrit\bigr)^{-1}$. 
\end{rmk}

\subsection{Counting Negative-Homogeneous Elements}

The next result indicates the 
complexity of the element counting map $c_F$ for an arbitrary noise.

\begin{prop}
\label{prop:bifalways}
Given any $\rho\in(0,2]$, $d\geq 2$ and $N\in \N$, there exists a noise
$\xi$ with negative H\"older regularity and $\tau\in \cF_F$ such that
$|\tau|_\fs=0$ and the stochastic fractional Allen--Cahn equation~\eqref{eq:AC}
is locally subcritical.
\end{prop}

\begin{proof}
Let $\tau:=\cI_\rho(\cI_\rho(\Xi)^{N})$ and observe that $\tau$ is 
constructible by the recursion~\eqref{eq:recurse1}--\eqref{eq:recurse2}. 
Calculating homogeneity yields
\benn
|\tau|_\fs = N(\rho+\alpha_0)+\rho.
\eenn
Now take $\alpha_0 = -(1+1/N)\rho$ to obtain $|\tau|_\fs=0$. Subcriticality
follows since $1>1-1/N^2$ implies $\rho>\alpha_0(1-N)/N$.
\end{proof}

Proposition~\ref{prop:bifalways} implies that the counting map $c_F$ can
actually produce a jump for every given fixed rational number if the noise,
still with negative H\"older regularity, is chosen suitably. To avoid this
significant complexity of \lq\lq bifurcations at any $\rho$\rq\rq\ (i.e.~new
elements appearing in the regularity structure upon parameter variation at any
$\rho$) it is reasonable to specialize the analysis to certain subclasses of
noise. Hence, we now particularise to the case of space-time white noise. Recall
that we say that an element $\tau$ of $\cF_F$ is of type
$(p,q,k)\in\N_0\times\N_0\times\N_0^{d+1}$ if it contains $p$ occurrences of
$\Xi$, $q$ occurrences of the integration operator $\cI_\rho$ and monomials of
total exponent $k$. The discussion in the previous subsection shows that if
$|\tau|_\fraks < 0$, then 
\begin{itemize}
\item 	$q$ is bounded by a number $q^\star$ of order
$(\rho-\rhocrit)^{-1}$ (cf.~\eqref{eq:pstar-qstar});
\item 	there is at most one value of $p$ for a given $q$, namely 
$p=1 + \lfloor \frac{N-1}{N} q \rfloor$. 
\end{itemize}
As described in~\cite{Hairer4}, each $\tau\in\cF_F$ can be represented by a
rooted tree with additional decorations. There are two types of edges, one of
them standing for the symbol $\cI_\rho$, and the other one representing
$\Xi$. Each vertex is decorated with an $\N^{d+1}$-valued label, representing
$X^k$, while multiplication of two symbols is denoted by concatenating the
corresponding trees at the root (\figref{fig_trees}a). 

If $\tau\in\cF_F$ is of type $(p,q,k)$, then it is represented by a decorated
tree with $p$ leaves and $p+q$ edges, where $p$ edges are of type $\Xi$ and
adjacent to a leaf, while $q$ edges are of type $\cI_\rho$. The tree has $p+q+1$
vertices, including the $p$ leaves, the root, and $q$ inner vertices. Each
vertex has at most degree $N+1$, except the root which has at most degree $N$. 

Since multiplication of symbols is commutative, the order of edges around any
vertex does not matter. Therefore the problem of estimating the element counting
map $c_F$ essentially amounts to counting, for each admissible $(p,q)$, the
number of \emph{non-homeomorphic} rooted trees satisfying the above constraints.
For counting purposes, it will also be useful to consider the~\emph{bare tree},
obtained by stripping a decorated tree of all $p$ edges of type $\Xi$ and
adjacent leaves, which has $q$ edges and $q+1$ vertices (\figref{fig_trees}b). 

\begin{figure}
\begin{center}
\begin{tikzpicture}[>=stealth',
root/.style={semithick,blue,circle,inner sep=0.05cm,fill=white,draw},
I node/.style={semithick,blue,circle,inner sep=0.05cm,fill=blue,draw},
Xi node/.style={semithick,teal,circle,inner sep=0.05cm,fill=white,draw},
I edge/.style={blue,semithick},
Xi edge/.style={teal,semithick,decorate, decoration={snake,segment
length=4pt,amplitude=2pt}},
scale=0.5
]

\draw[I edge] (-1,-1) -- (0,0) -- (1,-1) -- (1,-2) -- (0,-3);
\draw[I edge] (1,-2) -- (2,-3) -- (1,-4);
\draw[I edge] (2,-3) -- (3,-4);

%\draw[Xi edge] (-1,-1) -- (-2,-2);
\draw[Xi edge] (-1,-1) -- (-1,-2);
\draw[Xi edge] (0,-3) -- (0,-4);
\draw[Xi edge] (1,-4) -- (1,-5);
\draw[Xi edge] (3,-4) -- (3,-5);

\node[root] at (0,0) {};
\node[I node] at (-1,-1) {};
\node[I node] at (1,-1) {};
\node[I node] at (1,-2) {};
\node[I node] at (0,-3) {};
\node[I node] at (2,-3) {};
\node[I node] at (1,-4) {};
\node[I node] at (3,-4) {};

\node[Xi node] at (-1,-2) {};
\node[Xi node] at (0,-4) {};
\node[Xi node] at (1,-5) {};
\node[Xi node] at (3,-5) {};

\node[] at (-3,0.5) {{\bf (a)}};

\end{tikzpicture}
\hspace{15mm}
\begin{tikzpicture}[>=stealth',
root/.style={semithick,blue,circle,inner sep=0.05cm,fill=white,draw},
I node/.style={semithick,blue,circle,inner sep=0.05cm,fill=blue,draw},
Xi node/.style={semithick,red,circle,inner sep=0.05cm,fill=white,draw},
I edge/.style={blue,semithick},
Xi edge/.style={red,semithick,decorate, decoration={snake,segment
length=4pt,amplitude=2pt}},
scale=0.5
]

\draw[I edge] (-1,-1) -- (0,0) -- (1,-1) -- (1,-2) -- (0,-3);
\draw[I edge] (1,-2) -- (2,-3) -- (1,-4);
\draw[I edge] (2,-3) -- (3,-4);

\draw[white] (3,-4) -- (3,-5);
\node[white] at (3,-5) {};

\node[root] at (0,0) {};
\node[I node] at (-1,-1) {};
\node[I node] at (1,-1) {};
\node[I node] at (1,-2) {};
\node[I node] at (0,-3) {};
\node[I node] at (2,-3) {};
\node[I node] at (1,-4) {};
\node[I node] at (3,-4) {};

\node[] at (-3,0.5) {{\bf (b)}};

\end{tikzpicture}
\hspace{15mm}
\begin{tikzpicture}[>=stealth',
root/.style={semithick,blue,circle,inner sep=0.05cm,fill=white,draw},
I node/.style={semithick,blue,circle,inner sep=0.05cm,fill=blue,draw},
Xi node/.style={semithick,teal,circle,inner sep=0.05cm,fill=white,draw},
I edge/.style={blue,semithick},
Xi edge/.style={teal,semithick,decorate, decoration={snake,segment
length=4pt,amplitude=2pt}},
scale=0.5
]

\draw[I edge] (-1,-1) -- (0,0) -- (1,-1) -- (2,-2) -- (1,-3);
\draw[I edge] (2,-2) -- (3,-3) -- (2,-4);
\draw[I edge] (3,-3) -- (4,-4);

\draw[Xi edge] (1,-1) -- (0,-2);
\draw[Xi edge] (-1,-1) -- (-1,-2);
\draw[Xi edge] (1,-3) -- (1,-4);
\draw[Xi edge] (2,-4) -- (2,-5);
\draw[Xi edge] (4,-4) -- (4,-5);

\node[root] at (0,0) {};
\node[I node] at (-1,-1) {};
\node[I node] at (1,-1) {};
\node[I node] at (2,-2) {};
\node[I node] at (1,-3) {};
\node[I node] at (3,-3) {};
\node[I node] at (2,-4) {};
\node[I node] at (4,-4) {};

\node[Xi node] at (0,-2) {};
\node[Xi node] at (-1,-2) {};
\node[Xi node] at (1,-4) {};
\node[Xi node] at (2,-5) {};
\node[Xi node] at (4,-5) {};

\node[] at (-3,0.5) {{\bf (c)}};

\end{tikzpicture}
\vspace{-4mm}
\end{center}
\caption[]{{\bf (a)} Decorated tree representing
$\tau=\cI_\rho(\Xi)
\cI_\rho(\cI_\rho(\cI_\rho(\Xi)\cI_\rho(\cI_\rho(\Xi) ^2)))$, which is of type
$(4,7,0)$ and has degree vector $d(\tau)=(4,6,2)$; {\bf (b)} the associated bare
tree, whose degree vector is $d'(\tau)=(4,2,2)$. The root is denoted \tikz{\draw
node[semithick,draw=blue,fill=white,shape=circle,inner sep=0.05cm] {};}, while
each \tikz{\draw[thick,blue] (0,0) -- (0.7,0); \draw
node[semithick,draw=blue,fill=blue,shape=circle,inner sep=0.05cm] at (0.7,0)
{};} represents a symbol $\cI_\rho$, and each
\tikz{\draw[semithick,teal,decorate, decoration={snake,segment
length=4pt,amplitude=2pt}] (0,0) -- (0.7,0); \draw
node[semithick,draw=teal,fill=white,shape=circle,inner sep=0.05cm] at
(0.7,0) {};} represents a symbol $\Xi$. {\bf (c)} Another decorated tree,
representing $\tau=\cI_\rho(\Xi)
\cI_\rho(\Xi\cI_\rho(\cI_\rho(\Xi)\cI_\rho(\cI_\rho(\Xi) ^2)))$, which also
corresponds to the bare tree {\bf (b)}. However, such a tree cannot occur for
purely additive noise. 
}
\label{fig_trees}
\end{figure}

\begin{lem}
\label{lem:bare_decorated}
There is a one-to-one correspondence between bare and decorated trees.
Furthermore, any bare tree of maximal degree $N+1$ and root of maximal degree
$N$ represents an element constructible by the recursive
procedure~\eqref{eq:recurse1}--\eqref{eq:recurse2}. 
\end{lem}
\begin{proof}
We first prove the second claim, by induction on the number of vertices of the
tree. The trivial tree with one vertex and no edge represents the element ${\bf
1}$, which belongs to the model space, proving the base case. Consider now any
bare tree with maximal degree $N+1$ and maximal root degree $N$. If the root has
degree $1$, it corresponds to an element of the form $\cI_\rho(\tau')$, where
the tree representing $\tau'$ has maximal degree $N+1$ and maximal root degree
$N$, and thus belongs to the model space by induction hypothesis. If the root
has degree $2\leqs r\leqs N$, by cutting the root we obtain $r$ trees belonging
to the model space by induction hypothesis. Now the reverse of both operations
(adding an edge at the root or joining $r$ trees at their roots) are compatible
with the recursive procedure~\eqref{eq:recurse1}--\eqref{eq:recurse2}.

To prove the first claim, first observe that any bare tree can be made into an
admissible decorated tree by attaching an edge of type $\Xi$ to every leaf.
To prove that this is the only possibility, assume that we attach an edge of
type $\Xi$ to a vertex of the bare tree which has degree $2\leqs r\leqs N$
(\figref{fig_trees}c). This would mean that the corresponding element
contains the string $\Xi\cI_\rho(\Xi)^r$. However, one easily shows by 
induction that such elements cannot appear in the recursive
procedure~\eqref{eq:recurse1}--\eqref{eq:recurse2}. 
\end{proof}

\begin{rmk}
It is important to realise that the one-to-one correspondence between bare and
decorated trees only holds because we consider equations with purely additive
noise. For SPDEs of the form $\partial_t u = Lu + f(u) + g(u)\xi$, as those
considered for instance in~\cite{Hairer4}, decorated trees such as the one in
\figref{fig_trees}c can occur, meaning that several decorated trees can be
obtained from a given bare tree. 
\end{rmk}

\subsubsection{The case $N=2$}

Counting trees is simplest in the case $N=2$, because then it turns out that all
bare trees are either binary trees, or binary trees minus one edge. Recall
that a binary (rooted) tree is a tree in which each vertex except the leaves has
exactly two children. Thus the root has degree $2$, while all other vertices
have degree $3$ or $1$. 

We call~\emph{degree vector}, or simply~\emph{degree} of a tree the vector 
$(d_1,d_2,\dots)$ where $d_i$ denotes the number of vertices of degree $i$. The
degree vector of a  binary tree is of the form $(n+1,1,n-1)$ for some
$n\in\N$. We write $d(\tau) = (d_1,d_2,d_3)$ for the degree vector of the
decorated tree representing an element $\tau\in\cF_F$, and $d'(\tau) =
(d'_1,d'_2,d'_3)$ for the degree vector of the bare tree representing $\tau$.
The one-to-one correspondence described in Lemma~\ref{lem:bare_decorated}
implies that $d_1 = d'_1$, $d_3 = d'_3$ and $d_2-d'_2=p$ is the number of leaves
of the bare tree. 

\begin{prop}
\label{prop:trees_N2} 
Assume $N=2$, and let $\tau\in\cF_F$ be an element of type $(p,q,0)$ having
negative homogeneity. Then 
\begin{itemize}
\item 	if $q$ is even, then the bare tree representing $\tau$ is a binary
tree with $q+1$ vertices;
\item 	if $q$ is odd, then the bare tree representing $\tau$ is obtained by
removing one edge from a binary tree with $q+2$ vertices.
\end{itemize}
\end{prop}
\begin{proof}
For $N=2$, each vertex of a bare tree has at most degree $3$, and the root has
at most degree $2$. Furthermore, each leaf has degree $1$. Thus we have the
relations  
\begin{align}
\nonumber
 d'_1 + d'_2 + d'_3 &= q+1\;, \\
 d'_1 + 2d'_2 + 3d'_3 &= 2q\;.
 \label{eq:degrees} 
\end{align}
The second relation is due to the fact that by summing the degrees of all
vertices, each edge is counted exactly twice. 

If $q=2n$ is even, then~\eqref{eq:pq} implies $p=n+1$. If the root has degree
$2$, then $d'_1=p=n+1$ and the solution of the system~\eqref{eq:degrees} is
given by 
\begin{equation}
\label{eq:d123_qeven} 
 (d'_1, d'_2, d'_3) = (n+1, 1, n-1)\;.
\end{equation} 
This corresponds to a binary tree with $2n+1$ vertices and $n+1$ leaves, such as
$\cI_\rho(\Xi)^2$ if $n=1$. If the root has degree $1$, then $d'_1=p+1=n+2$, and
solving~\eqref{eq:degrees} yields $d'_2=-1$, which is not allowed. 

If $q=2n+1$ is odd, then~\eqref{eq:pq} yields again $p=n+1$. If the root has
degree $2$, then $d'_1=n+1$ and the solution of~\eqref{eq:degrees} is 
\[
 (d'_1, d'_2, d'_3) = (n+1, 2, n-1)\;.
\]
By adding one edge to the vertex of degree $2$ which is not the root, we 
obtain a binary tree with $2n+3$ vertices and $n+2$ leaves (for an example, see
\figref{fig_trees}b). 
Finally, if the root has degree $1$, then $d_1=n+2$ and 
\[
 (d'_1, d'_2, d'_3) = (n+2, 0, n)\;.
\]
This case is obtained by attaching one edge to the root, and a binary tree with
$2n+1$ vertices and $n$ leaves to the other end of this edge. For instance, if
$n=0$ one obtains the symbol $\cI_\rho(\Xi)$, while for $n=1$ one obtains
$\cI_\rho(\cI_\rho(\Xi)^2)$.
\end{proof}

The combinatorics of non-homeomorphic binary trees has been studied by
Otter~\cite{Otter1948}. 
The number of non-homeomorphic rooted binary trees with $n$ leaves is given by
the Wedderburn--Etherington number $w_n$. The first few of these numbers
(starting with $n=0$) are 
\[
 0, 1, 1, 1, 2, 3, 6, 11, 23, 46, 98, 207, 451, 983, 2179, 4850, 10905, 24631,
56011, \dots 
\]
(sequence {\tt A001190} in the On-Line Encyclopedia of Integer Sequences OEIS). 
In particular, it is known~\cite{Otter1948} that $w_n$ behaves asymptotically
like 
\[
 w_n \sim c_2 \frac{(\alpha_2^{-1})^n}{n^{3/2}}\;,
\]
where $\alpha_2 \decapprox{0.4026975}$ (OEIS sequence {\tt A240943}) is the
radius of convergence of the generating series $\sum_n w_n x^n$ and the
prefactor is $c_2 \decapprox{0.3187766}$ (OEIS sequence {\tt A245651}). 

\begin{thm}
\label{thm:cF2} 
For $N=2$, there exist constants $C_2^\pm$, depending only on $d$, such that the
number of negative-homogeneous elements satisfies
\begin{equation}
\label{eq:c0F_N2} 
 C^-_2 (\rho - \rhocrit)^{3/2} 
 \exp \biggl\{ \frac{\beta_2 d}{\rho-\rhocrit} \biggr\}
 \leqs c_F(2,d,\rho)
 \leqs C^+_2 (\rho - \rhocrit)^{3/2} 
 \exp \biggl\{ \frac{\beta_2 d}{\rho-\rhocrit} \biggr\}
\end{equation} 
for $\rhocrit < \rho \leqs 2$, 
where $\beta_2 = \frac89
\log(\alpha_2^{-1})\decapprox{0.8085063}$.
\end{thm}
\begin{proof}
Consider first the number $c^0_F(2,d,\rho)$ of negative-homogeneous elements
with trivial polynomial part, which are indexed by trees as given by
Proposition~\ref{prop:trees_N2}. We start by counting trees with an odd number
$q+1=2n+1$ of vertices, which are exactly rooted binary trees with $p=n+1$
leaves.
Condition~\eqref{eq:condition_pq} yields $1\leqs n<\frac12 q^\star$, so that the
total number of these trees is given by 
\[
\sum_{n=1}^{\lfloor q^\star/2 \rfloor} w_{2n+1} 
\asymp \frac{(\alpha_2^{-1})^{q^\star}}{(q^\star)^{3/2}}\;.
\]
The lower bound is obtained by considering only the last term of the sum,
while a matching upper bound is found by approximating the sum by an integral
(estimating separately the contribution of small and large $n$).
Taking into account the expression~\eqref{eq:pstar-qstar} for $q^\star$, we find
that the number of these trees obeys indeed~\eqref{eq:c0F_N2}. 

In addition, we have to count trees with an even number $q+2=2n+2$ of vertices
and $p=n+1$ leaves. In this case, Condition~\eqref{eq:condition_pq} yields
$0\leqs n<\frac12 (q^\star-p^\star) \leqs \frac14 q^\star$. Each of these $n$
yields $w_{2n+1}$ binary trees, and there are at most $2n+2$ places to attach
the additional edge. The total number of these trees is thus bounded above by 
\[
\sum_{n=0}^{\lfloor q^\star/4 \rfloor} (2n+2) w_{2n+1} 
\asymp \frac{(\alpha_2^{-1})^{q^\star/2}}{(q^\star)^{1/2}}
= q^\star(\alpha_2^{-1})^{-q^\star/2}
\frac{(\alpha_2^{-1})^{q^\star}}{(q^\star)^{3/2}}\;.
\]
Since $q\mapsto q(\alpha_2^{-1})^{-q/2}$ is bounded above, this number
satisfies the upper bound~\eqref{eq:c0F_N2} for an appropriate constant $C_2^+$.

To extend the result to elements with nontrivial polynomial part, note
that~\eqref{eq:tau_homogeneity} imposes $p\alpha_0+q\rho+|k|_\fs < 0$. If
$q=2n$, then $p=n+1$ and the condition becomes $n(3\rho-d) < \rho+d-2|k|_\fs$,
which translates into $n < \lambda(q^\star/2)$ for some $\lambda < 1$.
Since $|k|_\fs$ is bounded above by $(\rho+d)/2$ (cf.~the
proof of Theorem~\ref{thm:hF}), for each of these $n$, the number of choices to
add polynomial elements grows at most polynomially in $n$. The maximal value of
$n$ being only a fraction of $q^\star/2$, this does not modify the upper bound
on $c_F$. The same argument applies to odd $q$. 
\end{proof}

\subsubsection{The case $N>2$}

For general $N$, the most important r\^ole will be played by regular trees of
degree $N$, that is, trees in which each vertex except the leaves has exactly
$N$ children. The degree vector of such trees is of the form
$(d'_1,d'_2,\dots,d'_{N+1})=((N-1)n+1,0,\dots,0,1,n-1)$ for some $n\in\N$.  

\begin{prop}
\label{prop:trees_Ng2} 
Assume $N>2$, and let $\tau\in\cF_F$ be an element of type $(p,q,0)$ having
negative homogeneity. Then 
\begin{itemize}
\item 	if $q$ is a multiple of $N$, then the bare tree representing $\tau$
is a regular tree of degree $N$ with $q+1$ vertices;
\item 	otherwise, the bare tree representing $\tau$ is obtained by
removing $r$ edges, where $1\leqs r\leqs N-1$, from a regular tree of
degree $N$ with $q+r+1$ vertices.
\end{itemize}
\end{prop}
\begin{proof}
Similarly to the case $N=2$, we must have 
\[
 \sum_{j=1}^{N+1} d'_j = q+1\;, \qquad
 \sum_{j=1}^{N+1} jd'_j = 2q\;.
\]
If $q=Nn$ for some integer $n$, then $p=(N-1)n+1$. If the root has degree
larger than $1$, then $d'_1=p$ and we obtain the system 
\begin{align*}
d'_2 + \dots + d'_{N+1} &= n\;, \\
2d'_2 + \dots + (N+1)d'_{N+1} &= (N+1)n-1\;.
\end{align*}
Eliminating $d'_{N+1}$ we get $(N-1)d'_2 + \dots + d'_N = 1$. The only solution
is thus given by 
\begin{align*}
d'_1 &= (N-1)n+1\;, \\
d'_j &= 0  && \text{for $j=2,\dots,N-1$}\;, \\
d'_N &= 1\;, \\
d'_{N+1} &= n-1\;,
\end{align*}
which means that we have a regular tree of degree $N$. If the root has degree
$1$, then $d'_1=p+1$. Proceeding as above, we obtain $(N-1)d'_2 + \dots + d'_N =
-N+1$, which is not allowed. 

If $q=Nn+r$ for some $1\leqs r\leqs N-1$, then $p=(N-1)n+1+\lfloor
r-\frac{r}{N}\rfloor = (N-1)n+r$. If the root has degree larger than $1$, then
we obtain the system 
\begin{align*}
d'_2 + \dots + d'_{N+1} &= n+1\;, \\
2d'_2 + \dots + (N+1)d'_{N+1} &= (N+1)n+r\;,
\end{align*}
which yields $(N-1)d'_2 + \dots + d'_N = N+1-r$. This implies the bounds 
\begin{gather*}
 (N-1)n+1 \leqs d'_1 \leqs (N-1)(n+1) \;, \\
 2 \leqs \sum_{j=2}^N (N+1-j) d'_j \leqs N \;, \\
 n+1-N \leqs d'_{N+1} \leqs n\;.
\end{gather*} 
To obtain a regular tree, one has to attach $N+1-j$ edges to each vertex of
degree $j$ for $2\leqs j\leqs N-1$, and one edge to each vertex of degree $N$
which is not the root, which amounts to attaching at most $N-1$ edges. 

The last case occurs when $q=Nn+r$ and the root has degree $1$. Then a similar
computation yields $(N-1)d'_2 + \dots + d'_N = 1-r$, which imposes $r=1$ and 
the sum to vanish. This yields 
\begin{align*}
d'_1 &= (N-1)n+2\;, \\
d'_j &= 0  && \text{for $j=2,\dots,N$}\;, \\
d'_{N+1} &= n\;,
\end{align*}
and corresponds to a regular tree of degree $N$ attached to a single edge
originating in the root. 
\end{proof}

Although the combinatorics is a little bit more involved than in the case
$N=2$, when $\rho$ approaches $\rhocrit$ the vast majority of bare trees will be
regular trees of degree $N$. The number $w_n^{(N)}$ of non-homeomorphic regular
trees of degree $N$ with $n$ vertices has also been analysed
in~\cite{Otter1948}. It behaves asymptotically as 
\begin{equation}
\label{eq:wN_asymptotics} 
 w_n^{(N)} \sim c_N \frac{(\alpha_N^{-1})^n}{n^{3/2}}\;.
\end{equation} 
where $\alpha_N$ is the radius of convergence of the generating series. In
particular, $\alpha_3 \decapprox{0.3551817}$ and $\lim_{N\to\infty} \alpha_N
\decapprox{0.3383219}$. This yields the following result. 

\begin{thm}
\label{thm:Nbigger2}
For any $N>2$, there exist constants $C_N^\pm$, depending only on $N$ and $d$,
such that the number of negative-homogeneous elements
satisfies
\begin{equation}
\label{eq:c0F_N} 
 C^-_N (\rho - \rhocrit)^{3/2} 
 \exp \biggl\{ \frac{\beta_N d}{\rho-\rhocrit} \biggr\}
 \leqs c_F(N,d,\rho)
 \leqs C^+_N (\rho - \rhocrit)^{3/2} 
 \exp \biggl\{ \frac{\beta_N d}{\rho-\rhocrit} \biggr\}
\end{equation} 
for $\rhocrit < \rho \leqs 2$, 
where 
\begin{equation}
\label{eq:betaN} 
\beta_N = \frac{2N^2}{(N+1)^2}\log(\alpha_N^{-1})\;.
\end{equation} 
In particular,
$\beta_3 = \frac98 \log(\alpha_3^{-1}) \decapprox{1.164517}$.
\end{thm}
\begin{proof}
The proof of~\eqref{eq:c0F_N} follows along the lines of the proof of
Theorem~\ref{thm:cF2}. The number of trees with $q+1=Nn+1$ vertices is given by 
\[
\sum_{n=1}^{\lfloor q^\star/N \rfloor} w^{(N)}_{Nn+1} 
\asymp \frac{(\alpha_N^{-1})^{q^\star}}{(q^\star)^{3/2}}\;.
\]
Trees with $q+1=Nn+r+1$ vertices with $1\leqs r\leqs N-1$ have a negligible
effect on the asymptotics, because \eqref{eq:tau_homogeneity} 
and the constraint~\eqref{eq:pq} on $p$ yield 
\[
 n < \frac{\rho+d+2\kappa-2\rho r}{(N+1)(\rho-\rhocrit)-2\kappa(N-1)}
 = \lambda_r(\rho) \frac{q^\star}{N}\;,
\]
where 
\begin{equation}
\label{eq:lambdar} 
 \lambda_r(\rho) = 1 - \frac{2\rho r}{\rho+d+2\kappa}
 < \lambda_r(\rhocrit)
 = 1 - \frac{(N-1)r}{N+\kappa(N+1)/d}
 \leqs \frac{1}{N} + \cO(\kappa)\;,
\end{equation}
which is less than $1$ for $\kappa$ small enough.
\end{proof}

\section{Statistical Properties of the Model Space}
\label{sec:statreg}

Let us denote by $\cF_F^- = \{ \tau\in\cF_F \colon |\tau|_\fs < 0 \}$ the
basis of the negative-homogeneous sector of the model space. Since the cardinality 
of $\cF_F^-$ diverges as $\rho\searrow\rhocrit$, it is natural to consider
statistical properties observed when picking a tree uniformly at random in the
forest representing $\cF_F^-$. 

We thus consider the discrete probability space obtained by endowing $\cF_F^-$
with the uniform measure. We are interested in the distribution of various
random variables $Y:\cF_F^- \to \R$. Examples of such random variables are the
homogeneity $|\tau|_\fs$, the number of edges of the tree, its degree
distribution, its height and its diameter. 

\begin{rmk}
When stating results on the limit $\rho\searrow\rhocrit$, we will always
assume that the constant $\kappa>0$ defining $\alpha_0$ (cf.~\eqref{eq:alpha0})
is smaller than $\rho-\rhocrit$. This will simplify the expressions of various
limiting results, and is allowed since in practice we always
consider cases where $\rho>\rhocrit$.
\end{rmk}

\subsection{Tree Size Distribution}

The size of a bare tree can be measured by its number of edges $q$, which is
also the number of occurrences of the integration operator $\cI_\rho$ in
$\tau$. We will denote the corresponding random variable by an uppercase $Q$,
to avoid confusion with its values $q\in\{1,\dots,q^\star\}$. 

Proposition~\ref{prop:trees_Ng2} and the constraints on $p$ and $q$ imply that 
\begin{align}
\label{eq:lawQ0} 
\P\{ Q=q \} &= \frac{w^{(N)}_{q+1}}{c_F(N,d,\rho)}
&& \text{if $q \in N\N$ and $q\leqs q^\star$\;,} \\
\P\{ Q=q \} &\leqs \binom{q+N}{N} \frac{w^{(N)}_{q+N}}{c_F(N,d,\rho)}
&& \text{if $q \in N\N + r$ for $1\leqs r \leqs N-1$ 
and $q\leqs \lambda_r(\rho) q^\star$\;,} 
\nonumber
\end{align}
where $\lambda_r(\rho)$ is defined in~\eqref{eq:lambdar}.
The binomial coefficient in the second case bounds the number of ways of
pruning a regular tree of $N$ of its branches. The behaviour of the law of $Q$
as $\rho$ approaches $\rhocrit$ can be summarised as follows. Recall that
$\beta_N$ is defined in~\eqref{eq:betaN}.

\begin{prop}
\label{prop:lawQ}
There exists $\gamma=\gamma(N,d,\rho)>0$ satisfying
\begin{equation}
\label{eq:defgamma} 
 \lim_{\rho\searrow\rhocrit} \gamma(N,d,\rho) =
\beta_N d \biggl( 1-\frac1N \biggr)
\end{equation} 
such that
\begin{equation}
\label{eq:lawQ3} 
 \P\{Q \notin N\N\} \leqs \e^{-\gamma/(\rho-\rhocrit)}\;.
\end{equation}
Furthermore, $Q/q^\star$ satisfies the large-deviation estimate
\begin{equation}
\label{eq:lawQ2} 
 -\lim_{\rho\searrow\rhocrit} (\rho-\rhocrit) 
 \log \P \biggl\{ \frac{Q}{q^\star} \leqs x \biggr\} 
 = \beta_N d (1-x)
 \qquad \forall x\in [0,1]\;.
\end{equation}%
Finally, as $\rho\searrow\rhocrit$, one has
\begin{equation}
\label{eq:lawQ1} 
 \E\biggl( \frac{Q}{q^\star} \biggr) = 1 + \cO(\rho - \rhocrit)\;, 
 \qquad
 \Var\biggl( \frac{Q}{q^\star} \biggr) 
 = \cO\bigl((\rho - \rhocrit)^2\bigr)\;.
\end{equation}% 
\end{prop}
\begin{proof}
The proof of~\eqref{eq:lawQ3} draws on the fact that if $q$ is not a
multiple of $N$, then it cannot exceed $q_{\max} = q_{\max}(\rho) =
\lambda_1(\rho)q^\star$ where $\lambda_1(\rho)$ is defined
in~\eqref{eq:lambdar}. For this, it suffices to use the very rough upper bound
\[
 \P\{Q \notin N\N\} \leqs \sum_{q=1}^{q_{\max}}
 \binom{q+N}{N} \frac{w^{(N)}_{q+N}}{c_F(N,d,\rho)}
 \leqs q_{\max} \frac{(q_{\max}+N)^N}{N!} 
 \frac{w^{(N)}_{q_{\max}+N}}{c_F(N,d,\rho)}\;.
\]
Indeed, it follows from Theorem~\ref{thm:Nbigger2} and the
asymptotics~\eqref{eq:wN_asymptotics} of Wedderburn--Etherington numbers that 
\begin{align*}
(\rho-\rhocrit) \log c_F(N,d,\rho) 
&= \beta_N d + \cO\bigl( (\rho-\rhocrit)\log(\rho-\rhocrit) \bigr)\;, \\
(\rho-\rhocrit) \log(w^{(N)}_{q_{\max}+N})
&\leqs (\rho-\rhocrit) q_{\max}(\rho) \log(\alpha_N^{-1}) 
 + \cO\bigl( (\rho-\rhocrit)\log(\rho-\rhocrit) \bigr)\;.
\end{align*}
Using the definitions of $\beta_N$ and $q_{\max}(\rho)$, this yields 
\[
 (\rho-\rhocrit) \log \P\{Q \notin N\N\} \leqs 
 - \biggl( \beta_N d + (\rho - d)\frac{N}{N+1} 
 \log(\alpha_N^{-1}) \biggr)
 + \cO\bigl( (\rho-\rhocrit)\log(\rho-\rhocrit) \bigr)\;,
\]
from which~\eqref{eq:defgamma} and~\eqref{eq:lawQ3} follow upon taking the
limit $\rho\searrow\rhocrit$.

To prove the large-deviation estimate~\eqref{eq:lawQ2}, we write 
\begin{equation}
\label{eq:ldp_proof} 
 \P \biggl\{ \frac{Q}{q^\star} \leqs x \biggr\} 
 = \sum_{n=1}^{\lfloor xq^\star/N \rfloor} \P(Q=nN) 
 + \P \bigl\{ Q \leqs xq^\star , Q\notin N\N \bigr\}\;.
\end{equation} 
We claim that the sum is dominated by its last term. To see this, we rewrite it
as 
\[
 \sum_{n=1}^{\lfloor xq^\star/N \rfloor} \P\{Q=nN\}
 = \frac{w^{(N)}_{\lfloor xq^\star/N \rfloor N+1}}{c_F(N,d,\rho)} S 
 \qquad \text{with} \qquad 
 S = \sum_{n=1}^{\lfloor xq^\star/N \rfloor}
 \frac{w^{(N)}_{nN+1}}{w^{(N)}_{\lfloor xq^\star/N \rfloor N+1}}\;.
\]
Since $1 \leqs S \leqs \lfloor xq^\star/N \rfloor$, we have $\log(S) =
\cO(\log(\rho-\rhocrit))$. Thus the sum in~\eqref{eq:ldp_proof} obeys the
claimed large-deviation bound, owing to the fact that 
\[
 (\rho-\rhocrit) \log(w^{(N)}_{\lfloor xq^\star/N \rfloor N+1})
 = \beta_N d x + \cO\bigl( (\rho-\rhocrit)\log(\rho-\rhocrit) \bigr)
\]
combined with the previously obtained asymptotics of $c_F$. As for the second
term on the right-hand side of~\eqref{eq:ldp_proof}, it can be bounded above in
the same way as $\P\{Q \notin N\N\}$, with $q_{\max}$ replaced by $xq^\star$. As
a result, it is not larger than the large-deviation bound obtained for the sum.

The moment estimates~\eqref{eq:lawQ1} then follow from the
integration-by-parts formula 
\[
 0 \leqs 
 \E \biggl[ \biggl( 1-\frac{Q}{q^\star} \biggr)^p \biggr] 
 = \int_0^1 p y^{p-1} \P  \biggl\{ 1-\frac{Q}{q^\star} \geqs y \biggr\} \txtd
y\;,
\]
applied for $p\in\{1,2\}$, and the fact that 
$\P\{Q\leqs xq^\star\} \leqs \e^{-\gamma'(1-x)/(\rho-\rhocrit)}$ for some
$\gamma'>0$.
\end{proof}

This result shows in particular that as $\rho\searrow\rhocrit$, the random
variable $Q/q^\star$ converges to $1$ in $L^2$. In fact, it is easy to extend
the proof to show that it converges to $1$ in any $L^p$ with $p\geqs1$. This is
due to the fact that when $\rho$ is near $\rhocrit$, the overwhelming majority
of trees of negative-homogeneous elements have the maximal size $q^\star$.

The random variable $P$, counting the number of occurrences of $\Xi$ in $\tau$
as well as the number of leaves of the bare and decorated trees, is determined
by $Q$ owing to the relation~\eqref{eq:pq}. Indeed, note that~\eqref{eq:pq}
can be written 
\begin{equation}
 P = 1 + \frac{N-1}{N}Q -
 \biggl\{-\frac{Q}{N} \biggr\}
\end{equation} 
where $\{\cdot\}$ denotes the fractional part. In particular, \eqref{eq:lawQ3}
implies 
\begin{equation}
 \P \biggl\{ P = 1 + \frac{N-1}{N} Q \biggr\} 
 \geqs 1 - \e^{-\gamma/(\rho-\rhocrit)}\;.
\end{equation} 
This entails similar concentration properties for $P$ as for $Q$. 

\subsection{Homogeneity Distribution}

The random variable $\Hs(\tau) = |\tau|_\fs$, giving the homogeneity of
$\tau\in\cF_F^-$ can be expressed in terms of $Q$ and the random variable $\Xs$
giving the homogeneity of the polynomial part, i.e., $\Xs(\tau) = |k|_\fs$ if
$\tau$ is of type $(p,q,k)$. Indeed, using the fact that $p^\star\alpha_0 +
q^\star\rho=0$ and $p^\star-1 = (N-1)q^\star/N$ (cf.~\figref{fig_D0_D}), one
obtains $\rho/(-\alpha_0) = (N-1)/N + 1/q^\star$.
Hence~\eqref{eq:tau_homogeneity} yields  
\begin{equation}
 \Hs = \alpha_0 P + \rho Q + \Xs
 = -\alpha_0
 \biggl( \frac{Q}{q^\star} - 1 + \biggl\{-\frac{Q}{N} \biggr\} \biggr)
 + \Xs\;.
\end{equation} 
Note that $\Hs$ takes values in
$[\alpha_0,0)$, where $-\alpha_0$ converges to
$Nd/(N+1)$ as $\rho\searrow\rhocrit$, while $\Xs$ takes its values in a
finite subset of $\N_0 + \rho\N_0$ (see the proof of Theorem~\ref{thm:hF}).

\begin{prop}
\label{prop:lawH}
As $\rho\searrow\rhocrit$, one has 
\begin{equation}
 \E(\Hs) = \cO(\rho-\rhocrit)\;, \qquad
 \Var(\Hs) = \cO\bigl((\rho - \rhocrit)^2\bigr)\;.
\end{equation} 
Furthermore, $\Hs$ satisfies the large-deviation estimate 
\begin{equation}
\label{eq:ldp_homogeneity} 
 -\lim_{\rho\searrow\rhocrit} (\rho-\rhocrit) \log \P\{\Hs\leqs h\} 
 = \frac{N+1}N \beta_N (-h) 
 \qquad
 \forall h \in [\alpha_0,0]\;.
\end{equation} 
\end{prop}
\begin{proof}
First note that the probability $\P\{\Xs>0\} = \P\{\Xs\geqs\rho\}$ that $\tau$
admits a nontrivial polynomial part is exponentially small, as a consequence
of~\eqref{eq:tau_homogeneity}. Indeed, having $|k|_\fs > 0$ while $|\tau|_\fs <
0$ requires $Q$ to be bounded away from $q^\star$, so that we can apply the
large-deviation estimate~\eqref{eq:lawQ2}. The moment bounds thus follow
directly from those in Proposition~\ref{prop:lawQ}, while the large-deviation
bound is obtained by decomposing 
\begin{align*}
 \P\{\Hs \leqs h \} 
 ={}& \P\biggl\{ \frac{Q}{q^\star} - 1 \leqs \frac{h}{(-\alpha_0)}, Q\in
N\N ,\Xs=0\biggr\}
+ \P\biggl\{ \frac{Q}{q^\star} - 1 \leqs \frac{h - \Xs}{(-\alpha_0)}, Q\in
N\N ,\Xs\geqs\rho\biggr\} \\
 &{}+ \sum_{r=1}^{N-1} \P\biggl\{ \frac{Q}{q^\star} - \frac rN \leqs
\frac{h-\Xs}{(-\alpha_0)}, Q\in N\N + r\biggr\}\;.
\end{align*}
The large-deviation estimate~\eqref{eq:lawQ2} shows that the first term on the
right-hand side satisfies the bound~\eqref{eq:ldp_homogeneity} (note that
$\P(Q=q,\Xs=0)$ and $\P(Q=q)$ obey the same large-deviation bound). The
contribution of the other terms vanishes in the limit $\rho\searrow\rhocrit$.
\end{proof}

This result shows in particular that the random variable $\Hs$ converges to
$0$ in probability and in $L^2$ as $\rho \searrow \rhocrit$, and that the
probability that $\Hs = h < 0$ decays like
$\e^{-\gamma(-h)/(\rho-\rhocrit)}$ where $\gamma \to (N+1)\beta_N/N$ as $\rho
\searrow \rhocrit$. 

\subsection{Normalised Degree Distribution}
\label{ssec:degree_distrib} 

We define the normalised degree distribution of the bare and decorated trees
representing $\tau$ by the $2(N+1)$ random variables 
\begin{align}
\nonumber
 D'_j(\tau) &= \frac{d'_j(\tau)}{d'_1(\tau) + \dots + d'_{N+1}(\tau)}
 = \frac{d'_j(\tau)}{Q(\tau)+1}\;, 
 & j&=1,\dots,N+1\;, \\
 D_j(\tau) &= \frac{d_j(\tau)}{d_1(\tau) + \dots + d_{N+1}(\tau)}
 = \frac{d_j(\tau)}{P(\tau)+Q(\tau)+1}\;, 
 & j&=1,\dots,N+1\;.
 \label{eq:def_Dj} 
\end{align} 
The following result shows in particular that $(D'_1,\dots,D'_{N+1})$ converges
in probability and in $L^2$ to the deterministic limit 
\begin{equation}
 \biggl( \frac{N-1}{N}, 0, \dots, 0, \frac{1}{N} \biggr)
\end{equation} 
as $\rho\searrow\rhocrit$, while $(D_1,\dots,D_{N+1})$ converges (in the
same sense)
to 
\begin{equation}
 \biggl( \frac{N-1}{2N-1}, \frac{N-1}{2N-1}, 0, \dots, 0, \frac{1}{2N-1}
\biggr)\;.
\end{equation}
The difference between these random variables and their limits is
of order $\rho - \rhocrit$. 

\begin{prop}
\label{prop:lawDj}
The degree distributions satisfy 
\begin{align}
\nonumber
D'_1 &= \frac{N-1}{N} + R'_1\;, 
&
D_j &= \frac{N-1}{2N-1} + R_j\;, 
\qquad j=1,2\;, \\
D'_{N+1} &= \frac1N + R'_{N+1}\;, 
&
D_{N+1} &= \frac{1}{2N-1} + R_{N+1}\;,
\label{eq:lawD1} 
\end{align}
where the $R'_j$ are random variables satisfying 
\begin{equation}
 \E(R'_j) = \cO\bigl(\rho-\rhocrit\bigr)\;, \qquad 
 \Var(R'_j) = \cO\bigl((\rho-\rhocrit)^2\bigr)\;. 
\label{eq:lawD2} 
\end{equation} 
There exist constants $c>0$ and $\gamma>0$ such that 
\begin{equation}
\label{eq:lawD3} 
 \P\bigl\{|R'_j| > x(\rho-\rhocrit)\bigr\} \leqs \exp\biggl\{
-\frac{\gamma(1-c/x)}{\rho-\rhocrit} \biggr\} 
\end{equation} 
for $x\geqs c$. The $R_j$ are random variables satisfying
analogous relations. Furthermore, 
\begin{equation}
\label{eq:lawD4} 
 \E(D_j') = \cO\bigl((\rho-\rhocrit)\e^{-\gamma/(\rho-\rhocrit)}\bigr)\;, 
 \qquad
 \Var(D_j') = \cO\bigl((\rho-\rhocrit)^2\e^{-\gamma/(\rho-\rhocrit)}\bigr)
\end{equation} 
for $j=2,\dots,N-1$, and similarly for $D_j$ when $j=3,\dots,N-1$. 
\end{prop}
\begin{proof}
Consider the case of $D_1'$. It follows from the proof of
Proposition~\ref{prop:trees_Ng2} that 
\[
 D'_1 = \frac{N-1}{N} + \cO\biggl( \frac{1}{Q} \biggr)\;.
\]
To prove~\eqref{eq:lawD2}, it suffices to check that $1/Q$ has expectation of
order $1/q^\star$ and variance of order $1/(q^\star)^2$, which follows from the
fact that $Q$ is concentrated near $q^\star$. The tail estimate~\eqref{eq:lawD3}
follows from the large-deviation bound~\eqref{eq:lawQ2}. The expressions for the
$D_j$ are due to the fact that $D_1=D'_1$, $D_2=D'_2+P$ and $D_j=D'_j$ for
$j=3,\dots,N+1$, as a consequence of the one-to-one correspondence between bare
and decorated trees shown in Lemma~\ref{lem:bare_decorated}. The
bounds~\eqref{eq:lawD4} follow from the fact that the $D'_j$ and $D_j$ in
question vanish if $Q$ is a multiple of $N$ (see again the proof of
Proposition~\ref{prop:trees_Ng2}) together with~\eqref{eq:lawQ3}. 
\end{proof}

\subsection{Height and Diameter Distribution}

So far, we have only considered random variables which are either a function of
$Q$, or bounded in terms of $1/Q$. The distribution of a more general random
variable $Y$ can be expressed in terms of conditional expectations by 
\begin{equation}
\label{eq:expectationY} 
 \E(f(Y)) = \sum_{q=1}^{q^\star} \E(f(Y)\vert Q=q) \P\{Q=q\}
\end{equation} 
for any \lq\lq observable\rq\rq\ $f$ (e.g.\ $f(y)=y^p$). 
Examples of random variables with a nontrivial relation to $Q$ are the height
$H$ of a tree and its diameter $D$. The height of a rooted tree is defined as
the longest graph distance between the root and a leaf, while the diameter is
defined as the longest graph distance between leaves. Height and diameter of
nonhomeomorphic binary trees have been analysed
in~\cite{Broutin_Flajolet_2012}, yielding the following results (we consider
height and diameter of bare trees, but those of decorated trees are simply
obtained by adding $1$ or~$2$). 

\begin{prop}
\label{prop:lawHeight}
Assume $N=2$. Then
\begin{align}
\nonumber 
\E\bigl( \sqrt{\rho-\rhocrit}\, H \bigr)
&= \frac{4\sqrt{\pi d}}{3\lambda_2} + \cO(\rho-\rhocrit)\;,
&
\Var\bigl( (\rho-\rhocrit) H \bigr)
&= \frac{16\pi(\pi-3)d}{27\lambda_2^2} + \cO(\rho-\rhocrit)\;, \\
\E\bigl( \sqrt{\rho-\rhocrit}\, D \bigr)
&= \frac{16\sqrt{\pi d}}{9\lambda_2} + \cO(\rho-\rhocrit)\;,
&
\Var\bigl( (\rho-\rhocrit) D \bigr)
&= \frac{64(3-4\pi+\pi^2)d}{81\lambda_2^2} + \cO(\rho-\rhocrit)\;,
\end{align}
where $\lambda_2 \decapprox{1.1300337}$ is a constant related to the
generating series of the Wedderburn--Etherington numbers.
\end{prop}
\begin{proof}
Using the integration-by-parts formula
\[
 \E\biggl[ \biggl( \frac{Q}{q^\star} \biggl)^{1/2} \biggr] 
 = 1 - \int_0^1 \frac{1}{2(1-t)^{1/2}} 
 \P \biggl( \frac{Q}{q^\star} \leqs 1-t \biggr) \txtd t \;, 
\]
one checks that $\E(Q^{1/2}) = (q^\star)^{1/2} [1+\cO(\rho-\rhocrit)]$.
The result then follows from the moment estimates on $H$
in~\cite[Theorem~3]{Broutin_Flajolet_2012} and on $D$
in~\cite[Theorem~8]{Broutin_Flajolet_2012} (with $n=Q/2$), combined
with~\eqref{eq:expectationY} applied to $f(y)=y$ and $f(y)=y^2$. 
\end{proof}

This result shows that $H$ and $D$ are likely to be much smaller than $Q$, since
they are typically of order $1/\sqrt{\rho-\rhocrit}$ while $Q$ has order
$1/(\rho-\rhocrit)$. Their standard deviation, however, is of the same order as
their expectation, showing that they are much less concentrated than the other
random variables considered so far. 

\begin{rmk}
If furthermore the conditional
distribution satisfies a large-deviation principle 
\begin{equation}
 -\lim_{\rho\searrow\rhocrit} (\rho-\rhocrit) \log \P\biggl\{ Y\in A
\bigg\vert
\frac{Q}{q^\star} = x\biggr\} = \inf_{y\in A} I_Y(y|x)
\end{equation} 
(say for any interval $A\subset\R$), then~\eqref{eq:lawQ2} yields 
\begin{equation}
\label{eq:ldp} 
 -\lim_{\rho\searrow\rhocrit} (\rho-\rhocrit) \log \P\{ Y\in A \} 
 = \inf_{y\in A} \inf_{x\in[0,1]} \biggl[I_Y(y|x) + \beta_Nd(1-x)\biggr]\;.
\end{equation} 
The rate function $I_Y(y|x)$ can be interpreted as a relative negative entropy. 

For instance, in the case $N=2$, \cite[Theorem~5]{Broutin_Flajolet_2012}
provides a large-deviation estimate for $H$ with rate function $I_H$ which, when
applied to $Y=(\rho-\rhocrit)H$, yields 
\begin{equation}
 I_Y(y|x) \leqs \frac{4dx}{9} I_H \biggl( \frac{9y}{4dx} \biggr)\;.
\end{equation} 
Unfortunately, there is no explicit expression for $I_H$, which is expressed
in terms of the solution of a functional equation. It is to be expected,
however, that $x\mapsto I_Y(y|x)$ is decreasing, so that the infimum over $x$
in~\eqref{eq:ldp} is reached for $x=1$. In other words, $H$ is largest whenever
$Q$ is largest. 

A more subtle interplay between scales can occur when the noise is not purely
additive. Then there is no longer a one-to-one correspondence between bare and
decorated trees, so that the relative entropy of one with respect to the other
plays an important role.  
\end{rmk}

\section{Computation}
\label{sec:compute}

In this section, we provide a symbolic-algebra based computational analysis for
the model space $\cT_F$. We again exploit the viewpoint of model space elements
as directed rooted trees as considered in~\cite{Hairer4}. Recall that edges have
two types, $\Xi$ or $\cI_\rho$, corresponding to noise and integral kernel
respectively. Vertices also come in two types, terminal vertices with label $L$
marking a leaf with directed edge $\Xi$ pointed at the leaf and polynomial
vertices with label $X^k$ for a multiindex $k$.

\begin{figure}[htbp]
\centering \includegraphics[width=0.4\textwidth]{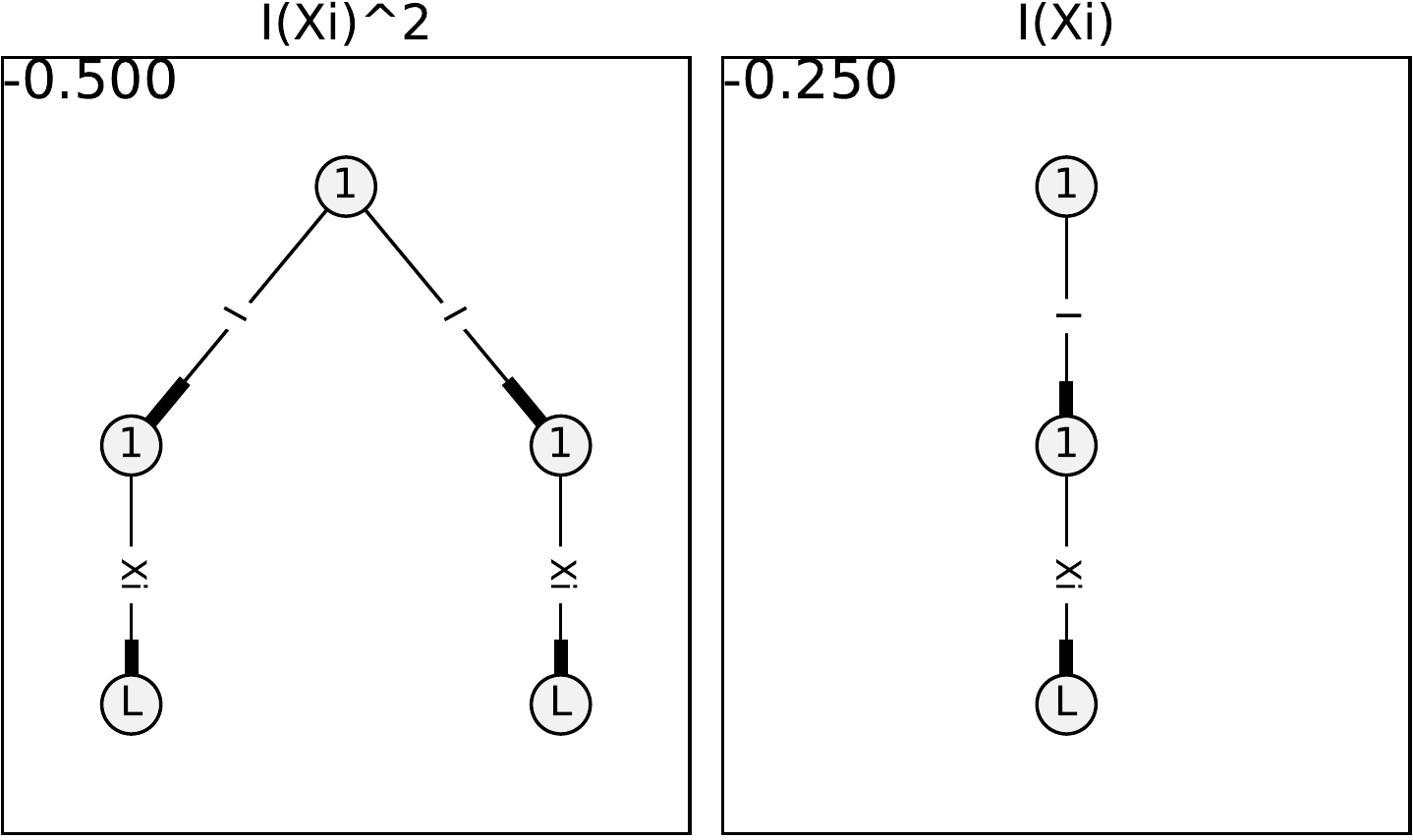}
\caption{\label{fig:01}Negative homogeneity elements in the model space
$\cT_F$ for a quadratic polynomial $f$ ($N=2$), in dimension $d=2$ and with
exponent $\rho=1.5$ for the SPDE~\eqref{eq:AC} with space-time white noise
forcing $\xi$. The top left corner for each element gives the homogeneity up to
a multiple of an arbitrarily small factor, i.e., $|\cI_\rho(\Xi)^2|_\fs =
-\frac12+\cO(\kappa)$ and $|\cI_\rho(\Xi)|_\fs = -\frac14 +\cO(\kappa)$ for
$0<\kappa\ll1$.}
\end{figure}

For an example consider Figure~\ref{fig:01}, which lists the elements of
negative homogeneity for the case $N=2$, $\rho=1.5$ and space-time white noise. 
For additive noise, these elements illustrate that the key building blocks have 
to be mixtures of integration against the kernel and taking powers. In the 
notation, we always suppress the trivial element ${\bf 1}$ but the element will 
be marked on the vertices of the trees to emphasize the product structure at
the vertices. 

The computation has been carried out in the package ReSSy (Regularity Structures
Symbolic Computation Package), which has recently been developed. The main
algorithm to compute the regularity structure elements provided the inputs
$N,d,\rho$ is the iterative procedure~\eqref{eq:recurse1}--\eqref{eq:recurse3}.
The algorithm has two main input parameters given by:
\begin{itemize}
 \item $\texttt{maxh}$ = maximum homogeneity of elements to keep after one
iteration,
 \item $\texttt{iter}$ = total number of fixed point iteration steps performed.
\end{itemize}
Of course, we cannot set $\texttt{maxh}=+\infty$ or to be arbitrarily large
since this would include all elements of the polynomial regularity structure.
For small to medium size regularity structures, it is possible to calculate all
elements of negative homogeneity but for very large structures we have to take
into account the fact that we may miss some elements if both algorithmic
parameters are not large enough. Here we decided to report the parameters for
each larger computation to guarantee for the reproducibility of results.

\subsection{Explicit Examples -- Negative Homogeneity Elements}
\label{ssec:negel}

\begin{figure}[tb]
\centering
\includegraphics[width=0.55\textwidth]{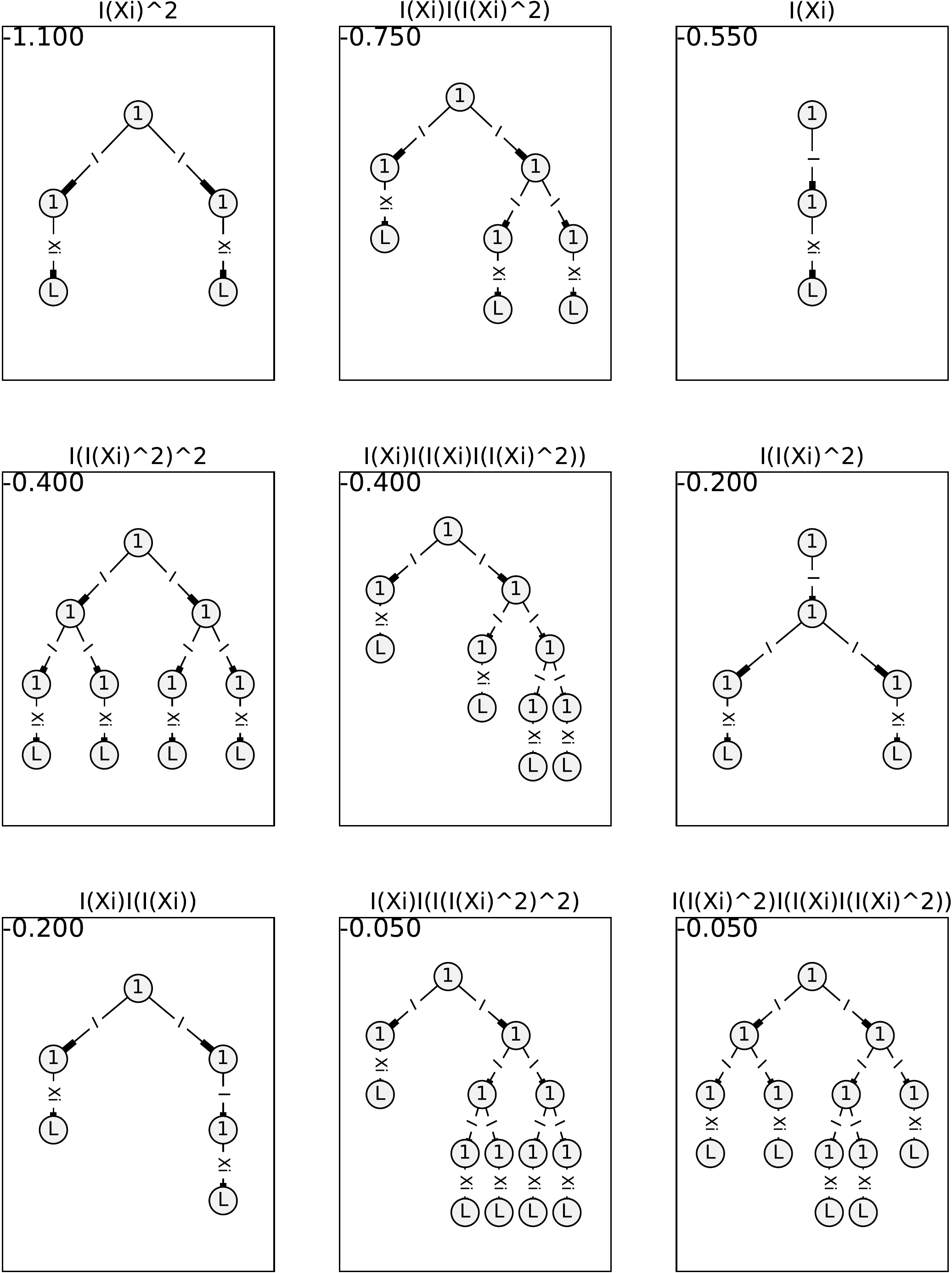}
\caption{\label{fig:02}List of all negative-homogeneous elements in the
basis $\cF_F$ of the model space $\cT_F$ for a quadratic polynomial $f$
($N=2$), for dimension $d=2$ and with exponent $\rho=0.9$ for the
SPDE~\eqref{eq:AC} with space-time white noise forcing $\xi$. The corresponding
pairs $(p,q)$ are shown in \figref{fig_N=2}.}
\end{figure}

In this section, we present a few more computational examples for space-time
white noise. As shown in Section~\ref{ssec:modelspace}, it makes no practical
sense to list all possible regularity structures. However, to build intuition,
it is important to explicitly compute with key examples. In this regard, we must
restrict the parameter space. Our main restriction is to consider 
\benn
d\in \{2,3\}\;,\qquad N\in\{2,3\}\;.
\eenn
The choice of dimension is motivated by the \lq\lq classical physical\rq\rq\
dimensions. $N>1$ is chosen since we are really interested in nonlinear
problems. $N\leq 3$ is motivated by classical normal form theory for ordinary
differential equations~\cite{Kuznetsov}, modulation/amplitude equations for
partial differential equations~\cite{Hoyle} and stochastic partial differential
equations~\cite{Bloemker}. In these contexts, one obtains normal forms near
criticality with $N\leq 3$ for the simplest codimension one bifurcations and
pattern-forming mechanisms. For the fractional Laplacian parameter $\rho$, it
makes sense to restrict to a regime 
\benn
\rho\in(0,2]\bigcap \left\{\rho>\rho_p> \rhocrit = d\frac{N-1}{N+1}\right\}
\eenn
where $\rho_p$ is a fixed number chosen as close as possible to the local
subcriticality boundary but also fixed so that the computations are still
possible in practice; indeed, we know from Section~\ref{ssec:modelspace} that
the model space grows very rapidly as we approach $\rhocrit$.

\begin{figure}[tb]
\centering \includegraphics[width=0.9\textwidth]{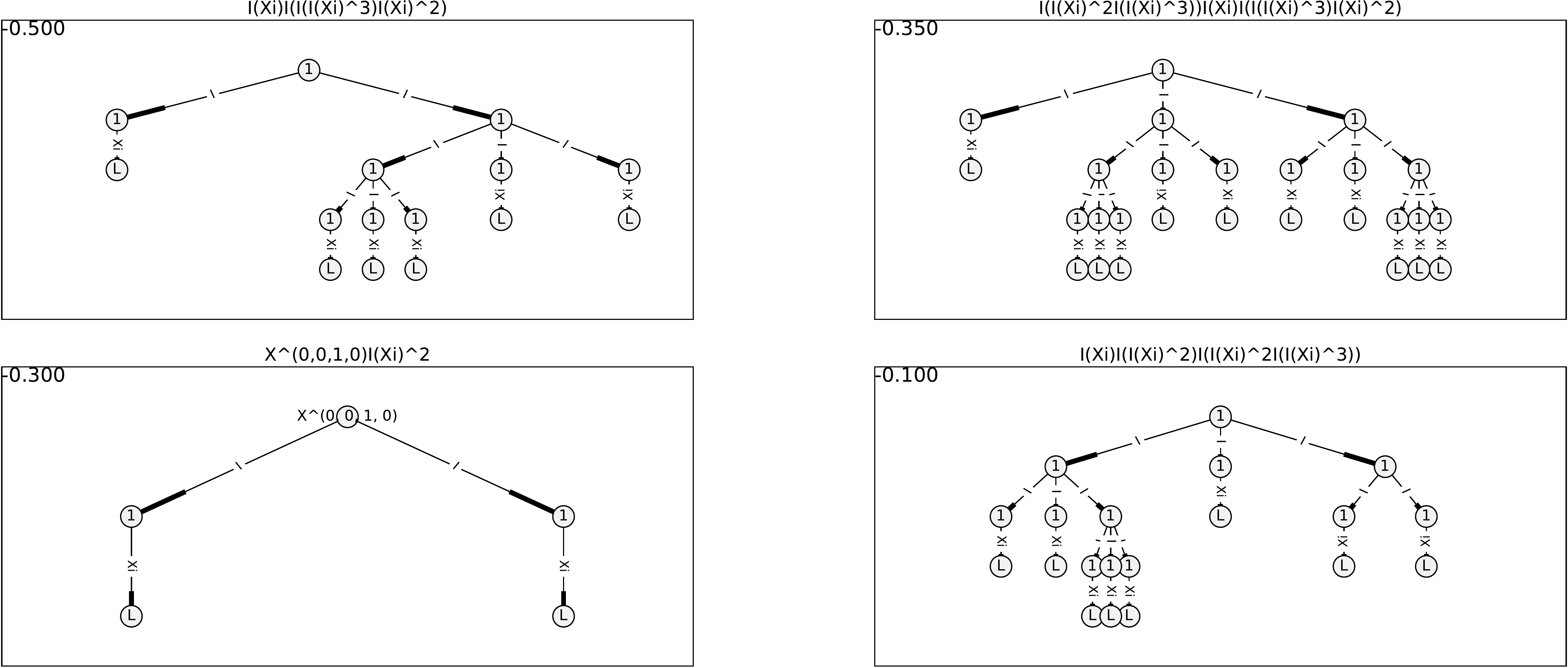}
\caption{\label{fig:03}Examples of some negative homogeneity elements in the
model space $\cT_F$ for a cubic polynomial $f$ ($N=3$), for dimension
$d=3$ and with exponent $\rho=1.7$ for the SPDE~\eqref{eq:AC} with space-time
white noise forcing $\xi$.}
\end{figure}

Figure~\ref{fig:02} shows an example for $d=2=N$ and $\rho=0.9$ listing all
elements of negative homogeneity. From the theory it is clear that the maximum
degree of a vertex must be $N+1 = 3$. One easily checks that the pairs $(p,q)$
counting the number of $\Xi$ and $\cI_\rho$ are compatible with the theory in
Section~\ref{ssec:homogeneities}, cf.~\figref{fig_N=2}. Furthermore, in
accordance with Proposition~\ref{prop:trees_N2}, all bare trees obtained by
pruning the edges of type $\Xi$ are either binary trees, or can be turned into
binary trees by adding one edge. 

Figure~\ref{fig:03}
shows just four elements for a more complicated case with $d=3=N$ and
$\rho=1.7$, where also nontrivial polynomial exponents appear; in this case,
there are at least 42 negative homogeneity elements in $\cF_F$
($\texttt{iter}=4$, $\texttt{maxh}=2.0$). In accordance with
Proposition~\ref{prop:trees_Ng2}, all pruned trees are either ternary trees, or
ternary trees pruned by one edge.

\subsection{Degree Distributions}
\label{ssec:dd}

The large dimension of the relevant sector of the model space may suggest
that it is too complex to understand in detail. However, as already shown in
Section~\ref{sec:statreg}, certain statistical properties become relevant as
$\rho\searrow\rhocrit$.

\begin{figure}[htbp]
\centering
\begin{overpic}[width=1\linewidth]{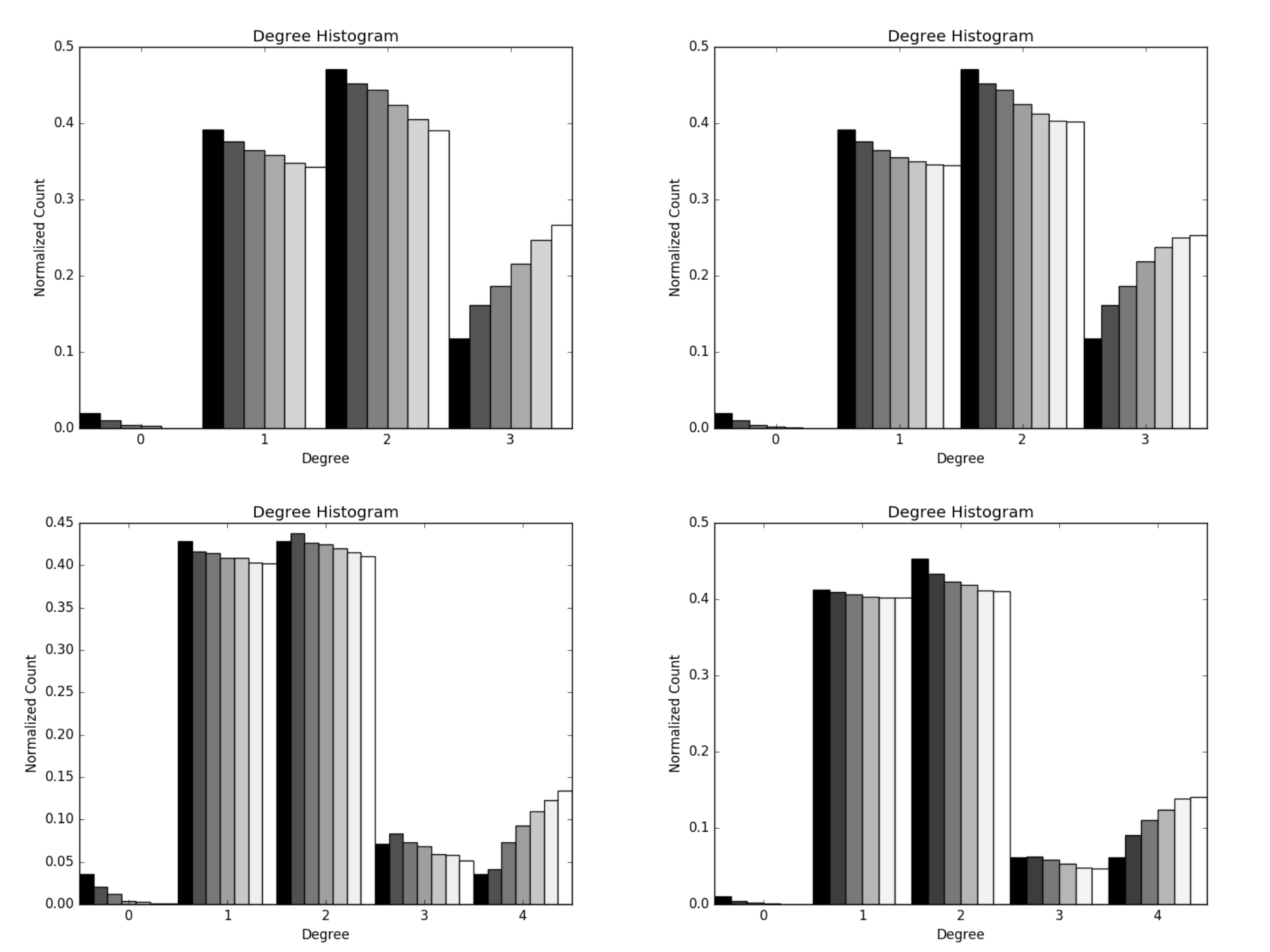}
\put(40,68){(a)}
\put(90,68){(b)}
\put(40,30){(c)}
\put(90,30){(d)}
\end{overpic}
\caption{\label{fig:04}Computation of the degree distribution for the parameter
values $N\in\{2,3\}$, $d\in\{2,3\}$ for negative homogeneous elements.
Each histogram is coded according to a greyscale, i.e., black corresponds to the
largest $\rho$-value used and $\rho=\rho_p$ corresponds to white. The degree
distribution is visualized as a histogram normalized to the total number of
elements of the regularity structure; note that we include the single trivial
element ${\bf 1}$ with degree zero as it provides an indication of the absolute
numbers in the normalized degree histogram. (a) Case $N=2=d$ (where
$\rhocrit=\frac23$), computed for $\texttt{maxh}=0.7$ and $\texttt{iter}=4$. The
exponents for the fractional Laplacian are $\rho\in
\{1.0,0.9,0.85,0.8,0.75,0.7\}$. (b) Case $N=2$, $d=3$, ($\rhocrit=1$), computed
for $\texttt{maxh}=1.0$, $\texttt{iter}=4$ and $\rho\in
\{1.4,1.3,1.25,1.2,1.15,1.1,1.08\}$. (c) Case $N=3$, $d=2$, ($\rhocrit=1$),
computed for $\texttt{maxh}=1.0$, $\texttt{iter}=3$ and $\rho\in
\{1.4,1.3,1.25,1.2,1.15,1.1,1.08\}$. (d) Case $N=3=d$, ($\rhocrit=\frac32$),
computed for $\texttt{maxh}=1.0$, $\texttt{iter}=3$ and $\rho\in
\{1.8,1.75,1.7,1.65,$ $1.6,1.59\}$.}
\end{figure}
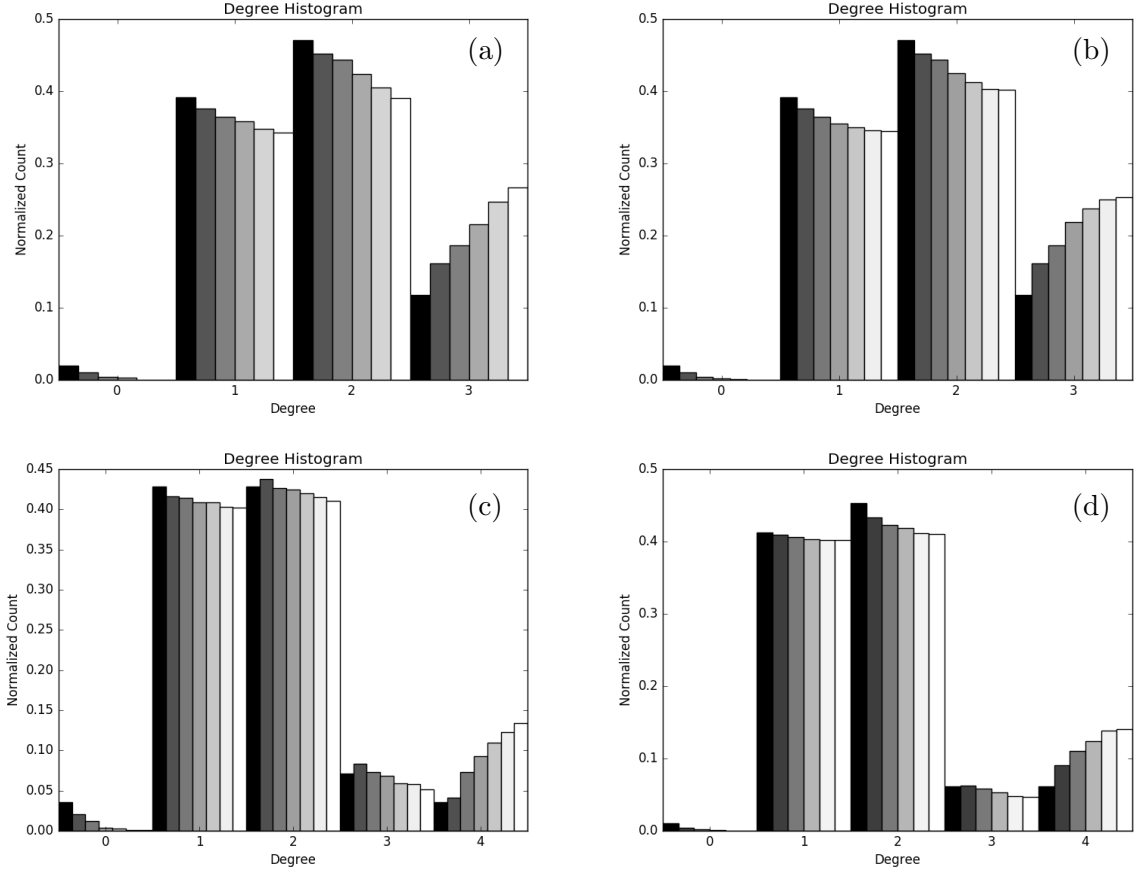

We start by considering the degree distribution $D$. Let $\cR$ denote the set of
rooted trees forming the set $\cF_F^-$ of negative-homogeneous basis
elements. Note that we can also view $\cR$ as a single graph. Then we define
\be
D(j;\rho)=\frac{\text{number of vertices of degree $j$ in $\cR$} }{\text{
total number of vertices in $\cR$}}=\frac{m_j}{m}\;. 
\ee
Note that this is slightly different from the $D_j$ defined in
Section~\ref{ssec:degree_distrib}, but both quantities are strongly related and
converge to the same deterministic limit as $\rho\searrow\rhocrit$. In
particular, it follows from~\eqref{eq:def_Dj} that $D(j;\rho) =
\E[D_j(P+Q+1)]/\E[P+Q+1]$.

Figure~\ref{fig:04} shows the degree distribution for the classical Allen--Cahn
case with $N,d\in\{2,3\}$ for different values of $\rho$ approaching the
subcriticality boundary at $d(N-1)/(N+1)=\rhocrit$; obviously we always consider
this limit as a limit from above. Although $m\ra +\I$ as $\rho\ra \rhocrit$, we
see that the degree distribution $D(k;\rho)$ has relatively stable features. The
trivial graph of the unit element ${\bf 1}$ explains the results at degree
zero. The results are compatible with Proposition~\ref{prop:lawDj}, which
implies that the degrees should converge to $(\frac13,\frac13,\frac13)$ for
$N=2$, and to $(\frac25,\frac25,0,\frac15)$ for $N=3$.

\begin{figure}[htbp]
\centering \includegraphics[width=0.8\textwidth]{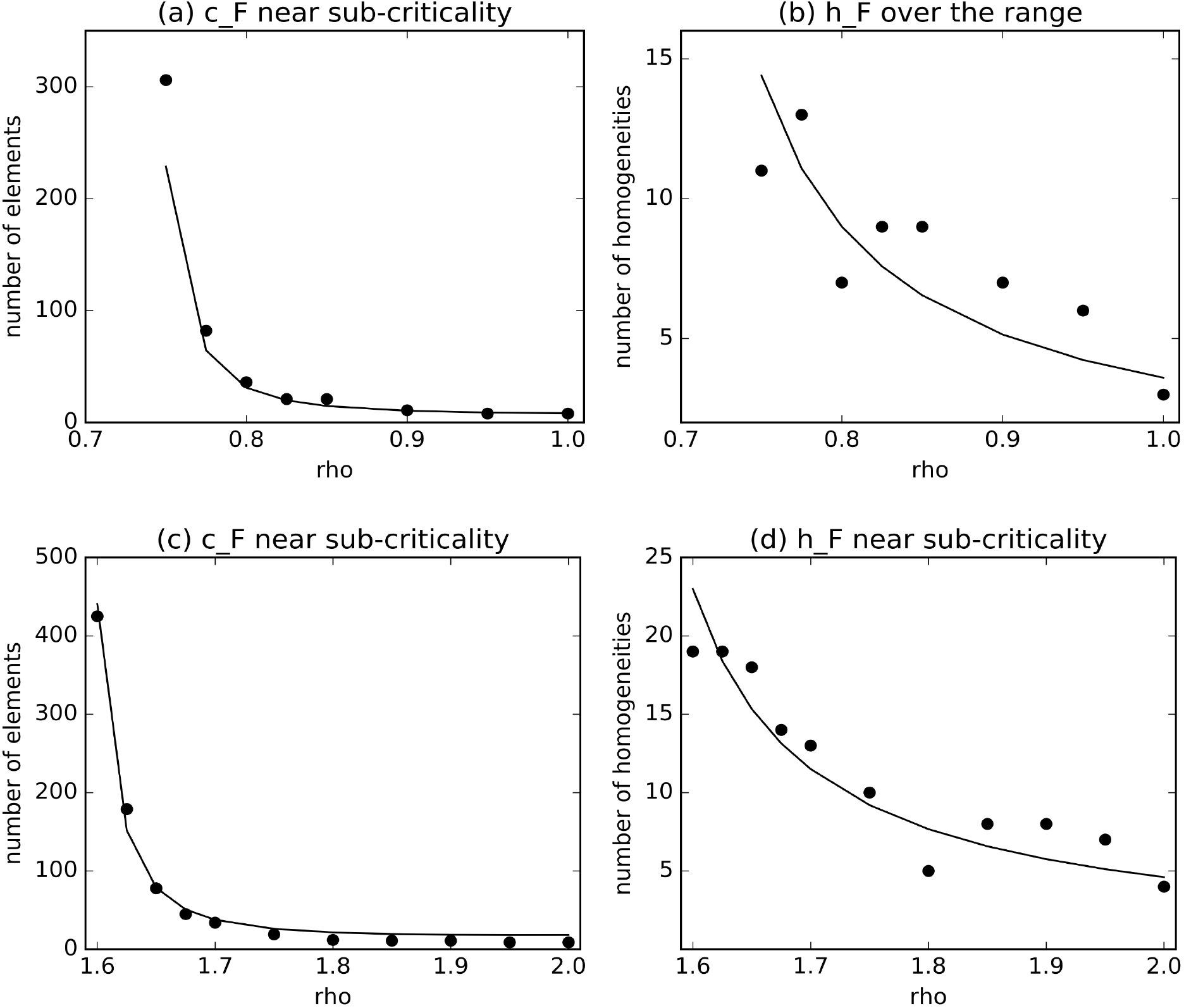}
\caption{\label{fig:07}Computation to analyse the functions $h_F$ and $c_F$.
Dots mark computational results while connected curves are theoretical scaling
laws plotted using a least-square fit to the theoretical scaling laws. (a)--(b)
Result for $N=2=d$ computed with $\texttt{maxh}=0.7$, $\texttt{iter}=5$.
(c)--(d) Result for $N=3=d$ computed for $\texttt{maxh}=0.8$,
$\texttt{iter}=4$.}
\end{figure}

Furthermore, we are also interested in the growth of the functions $c_F$ and
$h_F$. The results of this computation are shown in Figure~\ref{fig:07} for the
cases $N=2=d$ and $N=3=d$. The results are compatible with the asymptotic
growth in $(\rho-\rhocrit)^{-1}$ of $h_F$ obtained in Theorem~\ref{thm:hF} and
with the exponential growth of $c_F$ obtained in Theorems~\ref{thm:cF2}
and~\ref{thm:Nbigger2}.

We briefly summarize some computational observations. Obviously, the results for
the function $h_F$ are a lot easier to obtain computationally as we only need
a large enough sub-sample of the entire model space to count homogeneities while
for $c_F$, we have to count \emph{all} elements (up to homogeneities
equal to a multiple of $-\alpha_0$). Counting all elements
requires higher values for $\texttt{iter}$ and $\texttt{maxh}$, which can
substantially increase the computation time. The key computational bottleneck,
where the computation is slow, arises in the decision step whether during, or
after, the construction of an element, this element is already contained in
the model space from a previous iteration step. This step is unavoidable and
necessitates a comparison to previously computed elements. The larger the
individual graphs become, the more computationally intensive the computation may
be. Note that we already reduce the computation time significantly by doing
comparison by exclusion, e.g., first comparing the homogeneity, then comparing
the number of edges and nodes, etc., until we finally have to check for complete
graph isomorphism. Checking for graph isomorphism many times is extremely
expensive so it should be avoided as much as possible.

\subsection{Average Graph Properties}
\label{ssec:avggraph}

Although the degree distribution helps already to understand the approach
towards the subcriticality boundary, it is also very interesting to consider
other averaged graph properties.

\begin{figure}[htbp]
\centering \includegraphics[width=0.65\textwidth]{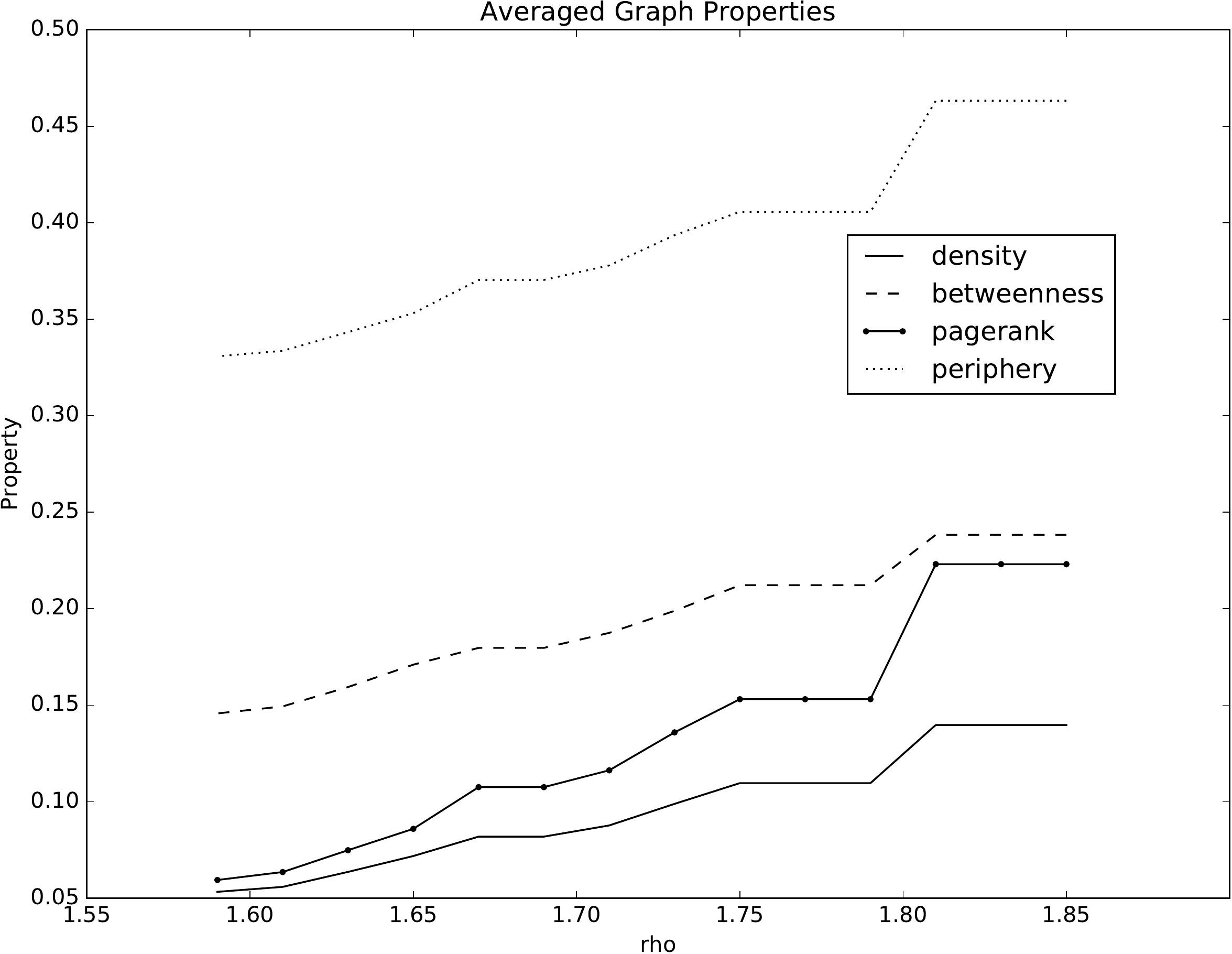}
\caption{\label{fig:05}Computation of different averaged graph properties
(vertical axis) for the regularity structure of the case $d=3=N$ for
different values of $\rho$ (horizontal axis); the computation has been carried
out for $\texttt{maxh}=1.5$ and $\texttt{iter}=3$. The properties are defined in
the text but one can already see that density $M_d$ and pagerank $M_{r}$ are
very close to zero at the subcriticality boundary; periphery $M_p$ and
betweenness $M_b$ are decreasing but are more likely to stabilize at the
subcriticality boundary $\rho=\frac32$ to finite nonzero values.} 
\end{figure}

As before, we let $\cR$ denote the set of rooted trees spanning the
negative-homogeneous sector $\cF_F^-$ of the model space and let $|\cR|=c_F$
denote the number of trees in $\cR$. Then we consider the following properties 
of $\cR$:
  
\begin{itemize} \item \textit{Density}: For an element
$R\in\cR$, let $n_R$ be the number of vertices in $R$ and $m_R$ be the number of
edges in $R$. Then the graph density averaged over $\cR$ is defined as
\be
M_d:=\frac{1}{|\cR|}\sum_{R\in\cR} \frac{m_R}{n_R(n_R-1)}\;.
\ee 
Since each $R$ is a tree, we have $n_R=m_R+1$ and thus $M_d$ is just the 
average of $1/n_R$, which we know to be of order $\rho-\rhocrit$. Hence, we
expect it to decay to zero as $\rho\ra \rho_c$.

\item \textit{Betweenness}: Let $\textnormal{path}(v_1,v_2)$ denote the number
of shortest paths between two vertices $v_1$ and $v_2$ and let
$\textnormal{path}(v_1,v_2|v)$ denote the number of shortest paths that also
pass through $v$; let $V_R$ denote the set of vertices of $R$. The betweenness
(or betweenness centrality) averaged over $\cR$ is defined as
\be
M_b:=\frac{1}{|\cR|}\sum_{R\in\cR}\frac{1}{n_R}\sum_{v\in R}
\sum_{v_1,v_2\in
V_R}\frac{\textnormal{path}(v_1,v_2|v)}{\textnormal{path}(v_1,v_2)}\;.
\ee 
Since each $R$ is a tree, $\textnormal{path}(v_1,v_2)=1$. However, betweenness
centrality for trees is still a property not fully
understood~\cite{FishKushwahaTuran} so it is of interest to just calculate it
here.

 \item \textit{PageRank}: Let $\textnormal{pagerank}(v)$ denote the pagerank of
a node computed according to~\cite{PageBrinMotwaniWinograd}, which measures the
importance of a vertex in a graph. Then the averaged PageRank is defined as
\be
M_r:= \frac{1}{|\cR|}\sum_{R\in\cR}\frac{1}{n_R}\sum_{v\in R}
\textnormal{pagerank}(v)\;.
\ee

 \item \textit{Periphery}: For $R\in \cR$ let $|\textnormal{ecc}_R|$ be the
number of vertices in $R$ with eccentricity equal to the diameter of $R$; recall
that the eccentricity of a vertex $v$ is the maximum distance from $v$ to all
other vertices and the diameter of a graph is the maximum eccentricity over all
nodes. Then the averaged periphery measure is defined as:
\be
M_p:= \frac{1}{|\cR|}\sum_{R\in\cR} |\textnormal{ecc}_R|\;.
\ee  
\end{itemize}

Note that for $M_b$ and $M_r$, we view each rooted tree as an undirected graph,
while the computations are for directed graphs for $M_d$ and $M_p$. The graph
properties are essentially coarse-grained summary statistics of the set of
rooted trees $\cR$ and represent different characteristics. Figure~\ref{fig:05}
shows a computation for the benchmark case fixing $d=3=N$ and leaving $\rho$ to
vary. The density $M_d$ decreases as expected. The averaged betweenness $M_b$
also decreases as $\rho$ decreases but seems to stabilize to a finite value,
i.e., there is a typical shortest path scale developing. The PageRank $M_r$ was
designed to measure the importance/connectedness of vertices and it also becomes
very small as $\rho\ra \frac32$. This indicates that although each rooted tree
is quite structured, it still grows individually in such a way to produce only a
few special or significant nodes so that the effect of $|\cR|$ increasing
eventually dominates. Quite interestingly, the averaged periphery measure $M_p$
also seems to have a well-defined finite value near subcriticality. $M_p$
essentially measures how many nodes are in the periphery of the trees and this
indicates that the growth of the rooted trees does follow a pattern still adding
a lot of smaller trees at the ends rather than maximizing connectedness.

\subsection{Homogeneity Distribution}
\label{ssec:homdist}

\begin{figure}[tb]
\centering \includegraphics[width=1\textwidth]{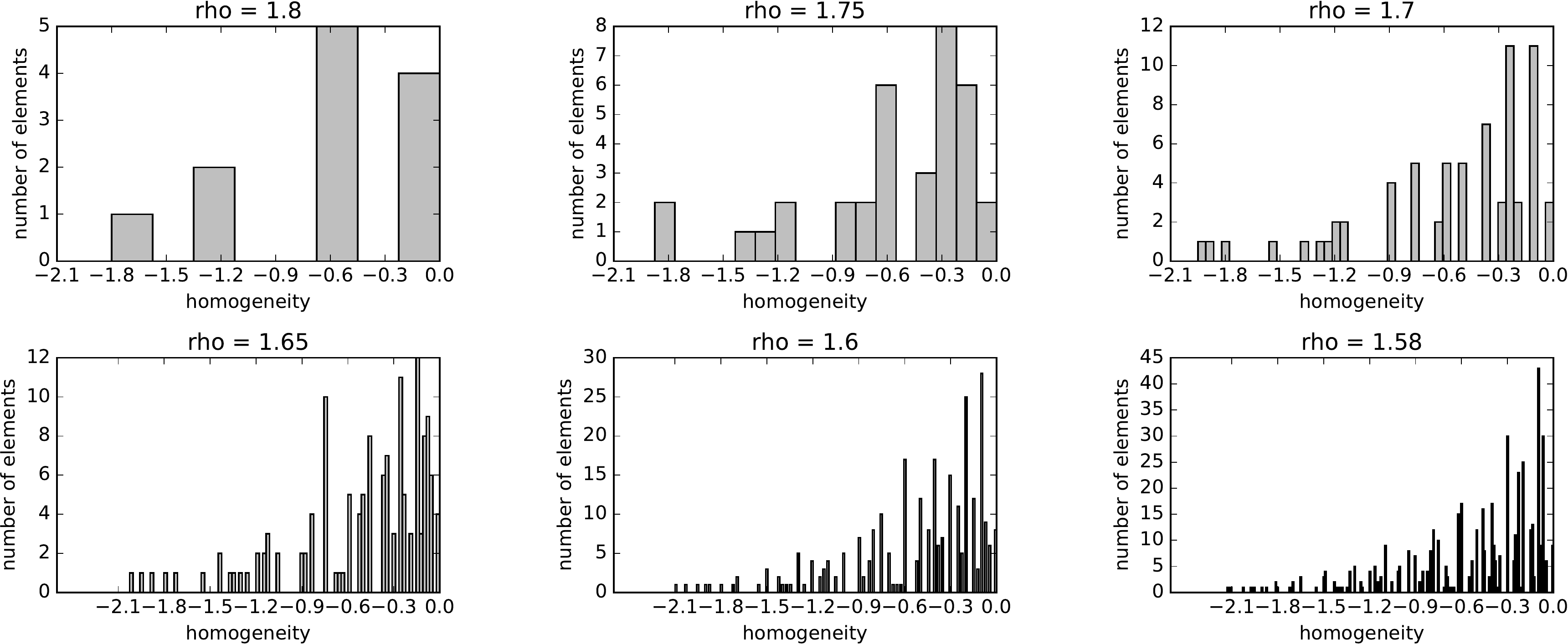}
\caption{\label{fig:06}Computation of homogeneity distributions for the
regularity structure $d=3=N$ for
different values of $\rho\in\{1.8, 1.75, 1.7, 1.65, 1.6, 1.58\}$, where each
histogram shows the number of elements in the regularity structure sorted
according to negative homogeneity but discarding arbitrarily small factors in
$|\cdot|_\fs=\cdot+\cO(\kappa)$, i.e., dropping the terms of order $\cO(\kappa)$
as $\kappa>0$ is arbitrarily small. The computation has been carried out for
$\texttt{maxh}=1.5$ and $\texttt{iter}=3$.}
\label{fig:homo_distrib} 
\end{figure} 

Another important measure for rooted trees in $\cF_F^-$ is their homogeneity,
i.e., we consider $\Hs(R) ={}|R|_\fs$ for $R\in\cR$ by just viewing the
rooted tree
as a symbol.

Figure~\ref{fig:06} again shows results for our benchmark case $d=3=N$ for
different values of $\rho$. Each histogram counts the number of elements of
a given homogeneity and each bin in the histogram is just one homogeneity
level. Note that we could have normalized the histograms by $|\cR|$ to obtain a
probability distribution, but the absolute counts are also informative and we
have scaled the vertical axis so that one already sees the normalized versions.
The results are consistent with Proposition~\ref{prop:lawH}, and in particular
with the large-deviation estimate~\eqref{eq:ldp_homogeneity} showing that the
probability of the homogeneity having a negative value $h$ behaves like
$\e^{-\kappa(-h)/(\rho-\rhocrit)}$. Another observation from Figure~\ref{fig:06}
is that upon decreasing $\rho$, new homogeneities appear in a very regular
fashion, at least reminiscent of the construction of functions with fractal
graphs. Indeed, if $\rho$ is rational, then only rational negative homogeneities
may appear by the iterative construction. However, the negative homogeneities
really do seem to fill out an entire domain of $\rho$. 

%\newpage

\bibliographystyle{plain}
\bibliography{BKfracAC}

\begin{thebibliography}{10}

\bibitem{AchleitnerKuehn1}
F.~Achleitner and C.~Kuehn.
\newblock Traveling waves for a bistable equation with nonlocal-diffusion.
\newblock {\em Adv. Differential Equat.}, 20(9):887--936, 2015.

\bibitem{AllenCahn}
S.M. Allen and J.W. Cahn.
\newblock A microscopic theory for antiphase boundary motion and its
  application to antiphase domain coarsening.
\newblock {\em Acta Metallurgica}, 27(6):1085--1905, 1979.

\bibitem{ArnoldSDE}
L.~Arnold.
\newblock {\em Random Dynamical Systems}.
\newblock Springer, Berlin Heidelberg, Germany, 2003.

\bibitem{BerglundDiGesuWeber}
N.~Berglund, G.~Di Ges{\`u}, and H.~Weber.
\newblock An {Eyring-Kramers} law for the stochastic {Allen-Cahn} equation in
  dimension two.
\newblock {\em arXiv:1604.05742}, pages 1--26, 2016.

\bibitem{BerglundKuehn}
N.~Berglund and C.~Kuehn.
\newblock Regularity structures and renormalisation of {FitzHugh-Nagumo SPDEs}
  in three space dimensions.
\newblock {\em Electron. J. Probab.}, 21(18):1--48, 2016.

\bibitem{Bloemker}
D.~Bl{\"{o}}mker.
\newblock {\em Amplitude Equations for Stochastic Partial Differential
  Equations}.
\newblock World Scientific, 2007.

\bibitem{Broutin_Flajolet_2012}
Nicolas Broutin and Philippe Flajolet.
\newblock The distribution of height and diameter in random non-plane binary
  trees.
\newblock {\em Random Structures Algorithms}, 41(2):215--252, 2012.

\bibitem{BrunedHairerZambotti}
Y.~Bruned, M.~Hairer, and L.~Zambotti.
\newblock Algebraic renormalisation of regularity structures.
\newblock {\em arXiv:1610.08468}, pages 1--84, 2016.

\bibitem{ChandraHairer16}
A.~Chandra and M.~Hairer.
\newblock An analytic {BPHZ} theorem for regularity structures.
\newblock {\em arXiv:1612.08138}, pages 1--113, 2016.

\bibitem{Chandra_Weber_15}
Ajay Chandra and Hendrik Weber.
\newblock Stochastic {PDE}s, regularity structures, and interacting particle
  systems.
\newblock {\em arXiv:1508.03616}, 2015.

\bibitem{ChenKimSong}
Z.-Q. Chen, P.~Kim, and R.~Song.
\newblock Heat kernel estimates for the {Dirichlet} fractional {Laplacian}.
\newblock {\em J. Eur. Math. Soc.}, 12:1307--1329, 2010.

\bibitem{CrossHohenberg}
M.C. Cross and P.C. Hohenberg.
\newblock Pattern formation outside of equilibrium.
\newblock {\em Rev. Mod. Phys.}, 65(3):851--1112, 1993.

\bibitem{Dafermos}
C.M. Dafermos.
\newblock {\em Hyperbolic Conservation Laws in Continuum Physics}.
\newblock Springer, 2010.

\bibitem{FishKushwahaTuran}
B.~Fish, R.~Kushwaha, and G.~Turan.
\newblock Betweenness centrality profiles in trees.
\newblock {\em arXiv:1607.02334}, pages 1--21, 2016.

\bibitem{Fisher}
R.A. Fisher.
\newblock The wave of advance of advantageous genes.
\newblock {\em Ann. Eugenics}, 7:353--369, 1937.

\bibitem{FrizHairer}
P.K. Friz and M.~Hairer.
\newblock {\em A Course on Rough Paths: With an Introduction to Regularity
  Structures}.
\newblock Springer, 2014.

\bibitem{Gardiner}
C.~Gardiner.
\newblock {\em Stochastic Methods}.
\newblock Springer, Berlin Heidelberg, Germany, 4th edition, 2009.

\bibitem{Gubinelli}
M.~Gubinelli.
\newblock Controlling rough paths.
\newblock {\em J. Funct. Anal.}, 216(1):86--140, 2004.

\bibitem{GubinelliImkellerPerkowski}
M.~Gubinelli, P.~Imkeller, and N.~Perkowski.
\newblock Paracontrolled distributions and singular {PDEs}.
\newblock {\em Forum Math. Pi}, 3:e6, 2015.

\bibitem{GubinelliTindel}
M.~Gubinelli and S.~Tindel.
\newblock Rough evolution equations.
\newblock {\em Ann. Probab.}, 38:1--75, 2010.

\bibitem{Hairer2}
M.~Hairer.
\newblock Solving the {KPZ} equation.
\newblock {\em Ann. Math.}, 178(2):559--664, 2013.

\bibitem{Hairer1}
M.~Hairer.
\newblock A theory of regularity structures.
\newblock {\em Invent. Math.}, 198(2):269--504, 2014.

\bibitem{Hairer4}
M.~Hairer.
\newblock The motion of a random string.
\newblock {\em arXiv:1605.02192}, pages 1--20, 2015.

\bibitem{Hairer3}
M.~Hairer.
\newblock Regularity structures and the dynamical {$\Phi^4_3$} model.
\newblock {\em arXiv:1508.05261}, pages 1--46, 2015.

\bibitem{HairerWeber}
M.~Hairer and H.~Weber.
\newblock Large deviations for white-noise driven, nonlinear stochastic {PDEs}
  in two and three dimensions.
\newblock {\em Ann. Fac. Sci. Toulouse Math.}, 24(1):55--92, 2015.

\bibitem{Hoyle}
R.~Hoyle.
\newblock {\em Pattern Formation: An Introduction to Methods}.
\newblock Cambridge University Press, 2006.

\bibitem{Bony}
J.-M-Bony.
\newblock Calcul symbolique et propagation des singularites pour les
  {\'e}quations aux d{\'e}riv{\'e}es partielles non lin{\'e}aires.
\newblock {\em Ann. Sci. Ec. Norm. Super.}, 4(14):209--246, 1981.

\bibitem{KardarParisiZhang}
M.~Kardar, G.~Parisi, and Y.C. Zhang.
\newblock Dynamic scaling of growing interfaces.
\newblock {\em Phys. Rev. Lett.}, 56(9):889--892, 1986.

\bibitem{KlagesRadonsSokolov}
R.~Klages, G.~Radons, and I.M. Sokolov.
\newblock {\em Anomalous Transport: Foundations and Applications}.
\newblock Wiley, 2008.

\bibitem{KolmogorovPetrovskiiPiscounov}
A.~Kolmogorov, I.~Petrovskii, and N.~Piscounov.
\newblock A study of the diffusion equation with increase in the amount of
  substance, and its application to a biological problem.
\newblock In V.M. Tikhomirov, editor, {\em {Selected Works of A. N. Kolmogorov
  I}}, pages 248--270. Kluwer, 1991.
\newblock {Translated by V. M. Volosov from Bull. Moscow Univ., Math. Mech. 1,
  1--25, 1937}.

\bibitem{Kuznetsov}
Yu.A. Kuznetsov.
\newblock {\em Elements of Applied Bifurcation Theory}.
\newblock Springer, New York, NY, 3rd edition, 2004.

\bibitem{Kwasnicki}
M.~Kwa{\'s}nicki.
\newblock Ten equivalent definitions of the fractional {Laplace} operator.
\newblock {\em arXiv:1507.07356}, pages 1--31, 2015.

\bibitem{MetzlerKlafter}
R.~Metzler and J.~Klafter.
\newblock The random walk's guide to anomalous diffusion: a fractional dynamics
  approach.
\newblock {\em Phys. Rep.}, 339(1):1--77, 2000.

\bibitem{Nagumo}
J.~Nagumo, S.~Arimoto, and S.~Yoshizawa.
\newblock An active pulse transmission line simulating nerve axon.
\newblock {\em Proc. IRE}, 50:2061--2070, 1962.

\bibitem{Otter1948}
Richard Otter.
\newblock The number of trees.
\newblock {\em Ann. of Math. (2)}, 49:583--599, 1948.

\bibitem{PageBrinMotwaniWinograd}
L.~Page, S.~Brin, R.~Motwani, and T.~Winograd.
\newblock The {PageRank} citation ranking: bringing order to the web.
\newblock {\em {http://ilpubs.stanford.edu:8090/422/1/1999-66.pdf}}, pages
  1--17, 1999.

\bibitem{DaPratoZabczyk}
G.~Da Prato and J.~Zabczyk.
\newblock {\em Stochastic Equations in Infinite Dimensions}.
\newblock Cambridge University Press, 1992.

\bibitem{Sato}
K.~Sato.
\newblock {\em L{\'{e}}vy Processes and Infinitely Divisible Distributions}.
\newblock CUP, 1999.

\end{thebibliography}

%\tableofcontents

\vfill
\goodbreak

\bigskip\noindent
{\small 
Nils Berglund \\ 
Universit\'e d'Orl\'eans, Laboratoire {\sc Mapmo}, 
{\sc CNRS, UMR 7349} \\
F\'ed\'eration Denis Poisson, FR 2964 \\
B\^atiment de Math\'ematiques, B.P. 6759\\
45067~Orl\'eans Cedex 2, France \\
{\it E-mail address: }{\tt nils.berglund@univ-orleans.fr}
}

\bigskip\noindent
{\small 
Christian Kuehn \\
Technical University of Munich (TUM) \\
Faculty of Mathematics \\
Boltzmannstr. 3 \\
85748 Garching bei M\"unchen, Germany \\
{\it E-mail address: }{\tt ckuehn@ma.tum.de}
}

\end{document}